\documentclass[letterpaper, 10 pt, conference]{ieeeconf}  

\usepackage{amsmath}
\usepackage{amssymb}
\usepackage{color}
\usepackage{enumerate}
\usepackage{lipsum}
\usepackage{mathtools}
\usepackage{cuted}
\usepackage{tabu}
\usepackage{float}

\usepackage{nameref}
\usepackage{varioref}
\usepackage{hyperref}
\usepackage{cleveref}

\usepackage{amsmath}                                   
\usepackage{amssymb} 
\usepackage{comment} 
\usepackage{rotating}

\usepackage{graphicx}
\usepackage{subcaption}

\usepackage{algorithm}
\usepackage{algpseudocode}
\algnewcommand{\Initialize}[1]{%
	\State \textbf{Initialization:}
	\Statex {\raggedright #1}
}

\newcommand{\EXP}[1]{\mathsf{E}\!\left[#1\right] }

\newtheorem{assumption}{Assumption}
\newtheorem{remark}{Remark}
\newtheorem{theorem}{Theorem}
\newtheorem{lemma}{Lemma}
\newtheorem{proposition}{Proposition}

\IEEEoverridecommandlockouts                              

\usepackage{booktabs}
\newcommand{\Real}{\mathbb{R}}

\newcommand{\green}[1]{{\color{black}#1}}

\usepackage[mathscr]{euscript}

\usepackage[utf8]{inputenc}
\usepackage[english]{babel}
\usepackage{diagbox}

\newcommand{\bx}{\mathbf{x}}

\newcommand{\bunit}{\mathbf{1}}

\newcommand{\bG}{\mathbf{G}}
\newcommand{\bxi}{\boldsymbol{\xi}}
\newcommand{\txsum}{\textstyle\sum}
\newcommand{\sF}{\mathscr{F}}
\newcommand{\sg}{\mathscr{G}}
\newcommand{\frm}{\tfrac{1}{m}}

\usepackage{array,multirow,graphicx}
 \usepackage{float}
\usepackage{caption}
\usepackage[dvipsnames,table,xcdraw]{xcolor}

\title{\LARGE \bf Distributed Randomized Block Stochastic Gradient Tracking Method}

\author{Farzad Yousefian  \and Jayesh Yevale  \and Harshal D. Kaushik 
\thanks{* This work is supported by the CAREER grant ECCS-1944500.}
\thanks{The authors are with the School of Industrial Engineering and Management,
        Oklahoma State University, Stillwater, OK 74078, USA.  Emails: {\tt\small   <farzad.yousefian, jayesh.yevale, harshal.kaushik>@okstate.edu. }}%
}%

\begin{document}

\maketitle
\thispagestyle{empty}
\pagestyle{empty}

\begin{abstract} We consider distributed optimization over networks where each agent is associated with a smooth and strongly convex local objective function. We assume that the agents only have access to unbiased estimators of the gradient of their objective functions. Motivated by big data applications, our goal lies in addressing this problem when the dimensionality of the solution space is possibly large and consequently, the computation of the local gradient mappings may become expensive. We develop a randomized block-coordinate variant of the recently developed distributed stochastic gradient tracking (DSGT) method. We derive non-asymptotic convergence rates of the order $1/k$ and $1/k^2$ in terms of an optimality metric and a consensus violation metric, respectively. Importantly, while block coordinate schemes have been studied for distributed optimization problems before, the proposed algorithm appears to be the first randomized block-coordinate gradient tracking method that is equipped with the aforementioned convergence rate statements. We validate the performance of the proposed method on the MNIST and a synthetic data set under different network settings. 
\end{abstract}

\section{Introduction}
We consider the distributed optimization problem 
\begin{equation}\label{prob:main}
\begin{aligned}
\min_{x \in \mathbb{R}^n}&  \quad \sum\nolimits_{i=1}^{m} \mathbb{E}[{f_i(x,\xi_i)}],
\end{aligned}
\end{equation}
where agents communicate over an undirected graph denoted by $\mathcal{G} =\left(\mathcal{N},\mathcal{E}\right)$, where $\mathcal{N}$ is the node set and $\mathcal{E} \subseteq \mathcal{N}\times \mathcal{N}$.  Here, $\xi_i \in \mathbb{R}^d$ denotes a local random variable. For the ease of exposition, we let $f_i(\bullet,\xi_i)$ and $f_i(\bullet)\triangleq \mathbb{E}[{f_i(\bullet,\xi_i)}]$ denote the stochastic and deterministic local objectives, respectively. We consider the following assumption.
\begin{assumption}\label{assum:problem}
For all $i \in \{1,\ldots,m\}$, function $f_i(\bullet)$ is $\mu$-strongly convex and $L$-smooth.
\end{assumption}
Distributed optimization problems over networks and in particular problem \eqref{prob:main}, find a wide range of applications in statistical learning, wireless sensor networks, and control theory~(see \cite{AngeliaReviewDO2018} and ~\cite{Lobel_11,Nedic_16,Aybat_17,Shi_15, SunLuHong2020,Ling_14,XinSahuKhanKar2019} as examples). In this paper, motivated by large-scale applications, we are interested in addressing problem \eqref{prob:main} in stochastic and high-dimensional settings. To this end, we assume that agents only have access to noisy local gradient mappings denoted by $\nabla f_i(\bullet, \xi_i)$ satisfying the following standard assumption. 
\begin{assumption}\label{assum:stoch_errors}
For all $i \in \{1,\ldots, m\}$, $\xi_i$ are independent from each other. Also, for all $i$ and $x \in \Real^n$, 
\begin{align*} 
&\EXP{\nabla f_i(x,\xi_i)|x} = \nabla f_i(x), \hbox{ and }\\
&\EXP{\parallel \nabla f_i(x,\xi_i)- \nabla f_i(x)\parallel^2 | x} \leq \nu^2 \quad \hbox{for some } \nu > 0.
\end{align*}
\end{assumption}
To account for high-dimensionality, our goal in this work lies in the development of a randomized block-coordinate gradient scheme where at each iteration, agents evaluate only random blocks of their stochastic gradient mapping.  To this end, throughout, we consider a block structure for $x$ given by $x = \left[x^{(1)}; \ldots;x^{(b)}\right]$ where $x^{(\ell)} \in \mathbb{R}^{n_i}$ denotes the $\ell$-th block-coordinate of $x \in \mathbb{R}^n$ and $\sum_{\ell=1}^b n_\ell= n$. 

\noindent \textbf{Existing methods and research gap.} Among the recent advancements in distributed optimization algorithms, gradient tracking methods have been recently studied. In these schemes, agents track the average of the global gradient mapping through communicating their estimate of the gradient locally with their neighbors in convex \cite{NedichOlshevskyShi2017,Pu_Shi_Nedic_20,Pu_Nedic_20,FarzadPushPull2020}, and nonconvex regime \cite{LorenzoScutari2016,QuLi2017, Scutari_19,ScutariSun2019}. In \cite{Pu_Shi_Nedic_20}, Push-Pull, G-Push-Pull algorithms and their variants are developed for addressing distributed optimization over directed graphs and a linear rate of convergence was established. Recently, a stochastic variant of gradient tracking methods has been developed in \cite{Pu_Nedic_20}, namely the DSGT method, where non-asymptotic convergence rates of the order $1/k$ and $1/k^2$ in terms of an optimality metric and a consensus violation metric were derived, respectively. 
Further, in \cite{LiZhengWangYanFengGuo2020}, integrating the ideas from DIGing \cite{NedichOlshevskyShi2017} and a fast incremental gradient method (SAGA) \cite{DefazioBachJulien2014}, S-DIGing algorithm is developed. 

In the aforementioned schemes, agents have to evaluate full-dimensional gradient vectors at each iteration of the method.  A popular avenue for addressing this issue is the class of block-coordinate schemes. Block-coordinate schemes, and specifically their randomized variants, have been widely studied in addressing optimization problems and games in deterministic~\cite{Nesterov2012,RicktarikTakac2014,ShwartzZhang2013,FarzadHarshalACC19,KaushikYousefianSIOPT2021} and stochastic regimes~\cite{Dang15,FarzadSetValued18,FarzadTAC19}. In randomized block schemes, at each iteration only a randomly selected block of the gradient mapping is evaluated, requiring significantly lower computational effort per iteration than the standard schemes. Although block-coordinate schemes have been studied for distributed optimization problems before \cite{Lorenzo_16,Notarnicol_20}, the convergence rate statements of randomized block gradient tracking methods are not yet established.  Inspired by the DSGT method~\cite{Pu_Nedic_20}, our goal in this paper lies in extending DSGT to a randomized block variant that is equipped with new non-asymptotic performance guarantees.     

{\begin{table*}[h]
\tiny
	\renewcommand\thetable{1}
	\captionsetup{labelformat=empty}
	\captionof{table}{{Table 1}: Comparison of this work with other recent gradient tracking schemes for  distributed optimization}
	\label{table:schemes_literature}
	\centering
	
		\begin{tabular}{ l  l l  l l l p{10cm} }
			\hline
			Reference  & Method & \begin{minipage}{1cm}Problem class\end{minipage} & \begin{minipage}{2cm} Network topology\end{minipage} & Problem formulation  & Rate(s)  \\ 
			\cmidrule(r){1-6}
			 \cite{Pu_Nedic_20} &  DSGT, GSGT   &    $f_i \in C^{1,1}_{\mu, L}$   &   Undirected  &  \begin{minipage}{5cm}$\underset{x\in \mathbb{R}^n}{\text{min}}\tfrac{1}{m}\sum_{i = 1}^m f_i(x) \triangleq  \mathbb{E}[{F_i(x,\xi_i)}]$ \end{minipage} & \begin{minipage}{3cm}  suboptimality: ${\cal O} \left( 1/k\right)$ \\ consensus: ${\cal O} \left(1/k^2\right)$\end{minipage}  \\
			\cite{Pu_Shi_Nedic_20}   &\begin{minipage}{1.7cm} Push-Pull \\ G-Push-Pull \end{minipage}&    $f_i \in C_{\mu,L}^{1,1}$   &  Directed  & {$\underset{x\in \mathbb{R}^n}{\text{min}}\sum_{i = 1}^m f_i(x)$}  &  linear \\
			\cite{Notarnicol_20}  & Block-SONATA & \begin{minipage}{1.5cm}$f_i \in C^{1,1}, \\r_\ell \in C_{0,0}^{0,0}$\end{minipage}   & Directed & \begin{minipage}{4.5cm}  $\underset{x}{\text{min}}  \sum_{i = 1}^m f_i(x) + \sum_{l = 1}^B r_l(x_l) \\ \text{s.t. } x_l\in K_\ell, \ \ell\in\{1,\dots,B\}$ \end{minipage}  & $-$ \\
			\cite{XinSahuKhanKar2019} &   S-AB &$f_i \in C^{1,1}_{\mu, L}$  &Directed & \begin{minipage}{5cm}   $\underset{x\in \mathbb{R}^n}{\text{min}}  \tfrac{1}{m} \sum_{i = 1}^m f_i(x) \triangleq \mathbb{E}[{F_i(x,\xi_i)}]$ \end{minipage} & linear    \\
			\cite{LiCenChenChi2020} &  Network-DANE &$f_i \in C^{1,1}_{\mu, L}$  &Undirected & \begin{minipage}{3.5cm}   $\underset{x\in \mathbb{R}^n}{\text{min}}  \tfrac{1}{N} \sum_{i = 1}^N \ell(x;z_i)$, \\ where $N=$ total samples, $z_i$ is the $i^{\text{th}}$ sample.\end{minipage} & linear    \\
			\cite{LiZhengWangYanFengGuo2020} &  S-Diging &$f_i \in C^{1,1}_{\mu, L}$  &Undirected & \begin{minipage}{5cm}   $\underset{x\in \mathbb{R}^n}{\text{min}}  \sum_{i = 1}^m f_i(x)\triangleq \mathbb{E}[{F_i(x,\xi_i)}]$ \end{minipage} & linear    \\
			\cite{LuZhangSunHong2019} &  GNSD &$f_i \in C^{1,1}$  &Undirected & \begin{minipage}{5cm}  
					$\underset{x}{\text{min}}  \tfrac{1}{m} \sum_{i = 1}^m f_i(x) \triangleq \mathbb{E}[{F_i(x,\xi_i)}]$\\
					s.t. \ $x_i = x_j , \ j\in \mathcal{N}(i), \  \forall i$
			 \end{minipage} & $\mathcal{O}\left(1/\sqrt{k}\right)$    \\
			\cmidrule(r){1-6}
			\begin{minipage}{0.1cm} \textbf{This}\\ \textbf{work} \end{minipage}  &    DRBSGT  & \begin{minipage}{1.5cm}$f_i \in {C^{1,1}_{\mu, L}}$ \end{minipage} & Undirected & \begin{minipage}{3.5cm}   $\underset{x\in \mathbb{R}^n}{\text{min}}  \ \sum_{i = 1}^m f_i(x) \triangleq \mathbb{E}[{f_i(x,\xi_i)}]$ \end{minipage} & \begin{minipage}{3cm}  suboptimality: ${\cal O} \left( 1/k \right)$ \\ consensus: ${\cal O} \left(1/k^2\right)$\end{minipage}\\
			\hline
		\end{tabular}
		\vspace{-.2in}
\end{table*}}
\noindent \textbf{Main contributions.} 
To highlight our contributions, we have prepared Table \ref{table:schemes_literature}. Our main contributions are as follows:

\noindent (i) We  develop a distributed randomized block stochastic gradient tracking algorithm (DRBSGT) for solving problem \eqref{prob:main}.  At each iteration, each agent evaluates only random blocks of its local stochastic gradient mapping. Importantly, we assume that the agents choose both their stochastic gradient mapping and their random block-coordinates independently from each other (see Algorithm~\ref{algorithm:RB-DSGT}).  

\noindent (ii) In Theorem \ref{thm:rate}, we derive a rate of $\mathcal{O}(1/k)$ on a suboptimality and $\mathcal{O}(1/k^2)$ on a consensus violation metric for Algorithm \ref{algorithm:RB-DSGT}. Importantly,  while DRBSGT generalizes DSGT to a randomized block variant, these rate statements match with those of DSGT, indicating that there is no sacrifice in terms of the order of magnitude of the  rate statements. 

\noindent (iii) To validate {the} theoretical claims, we compare the performance of our scheme with that of other existing gradient tracking schemes and provide preliminary results on different data sets and under different network assumptions. 

\noindent {\textbf{Outline.}} The rest of the paper is organized as follows. Section \ref{sec:alg_outline} includes the algorithm outline and some preliminaries. Section \ref{sec:conv_anal} provides the error bounds for the optimality and the consensus violation in Proposition \ref{prop:recursions}. The rate results are presented in Theorem \ref{thm:rate}. Section \ref{sec:num_expt} includes the numerical experiments. We provide concluding remarks in Section \ref{sec:conclusion}.

\noindent \textbf{Notation.}  
We let $x^*$ to denote the unique global optimal solution of problem \eqref{prob:main}. We let $\mathcal{N}(i)$ denote the set of neighbors of agent $i$, i.e., $\mathcal{N}(i) \triangleq \{j\mid (i,j) \in \mathcal{E}\}$. We use $[m]$ to denote $\{1,2,\ldots,m\}$ for any integer $m \geq 1$. Throughout, we let $\|\bullet\|$ denote the Euclidean norm and Frobenius norm of a vector and a matrix, respectively. We define $\mathbf{U}_\ell \in \mathbb{R}^{n\times n_\ell}$ for $\ell \in [b]$ such that $\left[\mathbf{U}_1, \ldots,\mathbf{U}_b\right] =\mathbf{I}_n $ where $\mathbf{I}_n$ denotes the $n\times n$ identity matrix. Note that we can write $x = \sum_{\ell =1 }^b \mathbf{U}_\ell x^{(\ell)}$, $\|\mathbf{U}_{\ell}x^{(\ell)}\|^2 = \|x^{(\ell)}\|^2$, 
\begin{align}\label{eqn:prop3_U_ell}
\textstyle\sum_{\ell=1}^b\|\mathbf{U}_{\ell}x^{(\ell)}\|^2 = \|x\|^2.
\end{align}
We consider the following notation throughout the paper.
$
\mathbf{x} := [x_1,x_2,\ldots,x_m]^T, \quad  \mathbf{y}  := [y_1,y_2,\ldots,y_m]^T \in \Real^{m \times n}, $\\
$\bar{x} := \tfrac{1}{m}\mathbf{1}^T \mathbf{x}   \in \Real^{1 \times n}, \quad  \bar{y} := \tfrac{1}{m}\mathbf{1}^T \mathbf{y}    \in \Real^{1 \times n},$

\noindent 
$ f(x) \triangleq \textstyle\sum\nolimits_{i=1}^{m} f_i(x), \quad \mathbf{f}(\mathbf{x}) \triangleq  \sum\nolimits_{i=1}^{m} f_i(x_i),$\\
$ f_i(x) \triangleq \EXP{f_i(x,\xi_i)\mid x},\quad \boldsymbol{\xi} := [\xi_1,\xi_2,\ldots,\xi_m]^T\in \Real^{m \times d},$\\
$\boldsymbol{\ell} := [\ell_1,\ell_2,\ldots,\ell_m]^T\in \Real^{m \times 1}, \\
\mathbf{G}(\mathbf{x},\boldsymbol{\xi}) \triangleq [\nabla f_1(x_1,\xi_1), \ldots, \nabla f_m(x_m,\xi_m)]^T,$\\
$ \mathbf{G}(\mathbf{x}) \triangleq \EXP{\mathbf{G}(\mathbf{x},\boldsymbol{\xi})\mid \mathbf{x}}= \notag [\nabla f_1(x_1), \ldots, \nabla f_m(x_m)]^T,$\\
$G(\mathbf{x},\boldsymbol{\xi}) \triangleq \tfrac{1}{m}\mathbf{1}^T\mathbf{G}(\mathbf{x},\boldsymbol{\xi}) =\tfrac{1}{m}\textstyle\sum_{i=1}^m\nabla f_i(x_i,\xi_i),$

\noindent $  G(\mathbf{x}) \triangleq \EXP{G(\mathbf{x},\boldsymbol{\xi})\mid \mathbf{x}}, \quad \sg(x) \triangleq G(\mathbf{1}x^T) = \tfrac{1}{m}\nabla f(x).$
\section{Algorithm Outline}\label{sec:alg_outline}
\begin{algorithm*}[t]
  \caption{Distributed Randomized Block Stochastic Gradient Tracking (DRBSGT)}\label{algorithm:RB-DSGT}
    \begin{algorithmic}[1]
    \State\textbf{Input:} Agents choose $\gamma_{0} > 0$ the weight matrix $\mathbf{W}$.  For all $i \in [m]$, agent $i$ chooses a random initial point $x_{i,0} \in \mathbb{R}^n$
    \State For all $i \in [m]$, agent $i$ generates samples $\xi_{i,0}$ and $\ell_{i,0}$ and $y_{i,0}^{(\ell_{i,0})}:= \nabla^{\ell_{i,0}} f_i(x_{i,0},\xi_{i,0}) $ and $y_{i,0}^{(\ell)}:= 0$ for all $\ell \neq \ell_{i,0}$. 
    
    \For {$k=0,1,\ldots,$}
    		 \State For all $i \in [m]$, agent $i$ does the following update : 
    		     		$
    		  x_{i,k+1} := \ 
    		  \txsum\nolimits_{j=1}^{m} W_{ij}\left(x_{j,k}  - \gamma_{k}y_{j,k}\right).
$  

\State For all $i \in [m]$, agent $i$ generates realizations of the random variables $\xi_{i,k+1}$ and $\ell_{i,k+1}$.
\State For all $i \in [m]$, agent $i$ does the following update: 
\begin{align*}
y_{i,k+1}^{(\ell)} &:= \left\{\begin{array}{ll}
        \sum\nolimits_{j=1}^m W_{ij}y_{j,k}^{(\ell)} +\nabla^{(\ell)}f_i(x_{i,k+1},\xi_{i,k+1})-\nabla^{(\ell)} f_i(x_{i,k},\xi_{i,k}), & \text{if } \ell =\ell_{i,k+1} = \ell_{i,k} \\
         \sum\nolimits_{j=1}^m W_{ij}y_{j,k}^{(\ell)} +\nabla^{(\ell)} f_i(x_{i,k+1},\xi_{i,k+1}), & \text{if } \ell = \ell_{i,k+1}\neq \ell_{i,k}  \\
       \sum\nolimits_{j=1}^m W_{ij}y_{j,k}^{(\ell)} -\nabla^{(\ell)} f_i(x_{i,k},\xi_{i,k}) & \text{if } \ell =  \ell_{i,k} \neq \ell_{i,k+1} \\
        \sum\nolimits_{j=1}^m W_{ij}y_{j,k}^{(\ell)} , & \text{if }  \ell \neq  \ell_{i,k+1},  \ell \neq \ell_{i,k}.
        \end{array}\right.
\end{align*}
   \EndFor
   \end{algorithmic}
\end{algorithm*}

The outline of the proposed scheme is presented in Algorithm \ref{algorithm:RB-DSGT}. At each iteration, agents evaluate random blocks of their local stochastic gradients under the assumption below.
\begin{assumption}\label{assum:rand_block}
For $k\geq 0$ and $i \in [m]$, let $\ell_{i,k} \in [b]$ be generated from a discrete uniform distribution, i.e., $\textrm{Prob}(\ell_{i,k} = \ell) = b^{-1}$ for all $\ell \in [b]$. 
\end{assumption}
\begin{remark}
Note that in Assumption \ref{assum:rand_block}, the block $\ell_{i,k}$ for each agent $i$ is independently selected. Also, note that this block selection is independent from the random variables $\xi_i$.
\end{remark}
Algorithm \ref{algorithm:RB-DSGT} can compactly be written as
\begin{alignat}{3}\label{eqn:udpate_rules_compact}
&\mathbf{x_{k+1}} &&= \mathbf{W}(\mathbf{x_{k}} - \gamma_k\mathbf{y_{k}}), \notag\\
&\mathbf{y_{k+1}} &&= \mathbf{Wy_{k}} + b^{-1}\left( \mathbf{G}(\mathbf{x}_{k+1},\boldsymbol{\xi}_{k+1})- \mathbf{e}_{k+1}\right) \notag\\
& && - b^{-1}\left( \mathbf{G}(\mathbf{x}_{k},\boldsymbol{\xi}_{k})- \mathbf{e}_{k}\right),
\end{alignat}
where $\mathbf{W}$ is a doubly stochastic weight matrix and $\gamma_{k}$ is a nonincreasing strictly positive stepsize sequence.  Similar to~\cite{Pu_Nedic_20}, we consider the following assumptions. 
\begin{assumption}\label{assum:Stochastic_W}
The weight matrix $\mathbf{W}$ is doubly stochastic and we have $w_{i,i}>0$ for all $i \in [m]$. 
\end{assumption}
\begin{assumption}\label{assum:Graph}
Let the graph $\mathcal{G}$ corresponding to the communication network be undirected and connected.
\end{assumption}
Throughout, we show the history of the scheme for $k\geq 1$ as $
\mathscr{F}_k \triangleq \cup_{i=1}^m \{x_{i,0}, \ell_{i,0},\xi_{i,0},\ldots,\ell_{i,k-1},\xi_{i,{k-1}}\},$ with $\mathscr{F}_0 \triangleq \cup_{i=1}^m \{x_{i,0}\}$. 
We define the stochastic errors of the randomized block-coordinate scheme as follows.
\begin{align}\label{def:e_ik}
e_{i,k} &\triangleq \nabla f_i(x_{i,k},\xi_{i,k}) - b\mathbf{U}_{\ell_{i,k}}\nabla^{\ell_{i,k}} f_i(x_{i,k},\xi_{i,k}), \\
\mathbf{e}_{k}&\triangleq [e_{1,k},e_{2,k},\ldots,e_{m,k}]^T  \in \Real^{m \times n},\notag \\
\bar{e}_{k} &\triangleq \tfrac{1}{m}\mathbf{1}^T\mathbf{e}_k =\frm\txsum_{i=1}^me_{i,k}.\notag
\end{align} 
We show some key properties of the randomized errors.
\begin{lemma}\label{lem:random_block_error}
We have for all $i \in [m]$ and $k\geq 0$ 

\noindent (a) $\EXP{e_{i,k}\mid \mathscr{F}_k}= \EXP{\bar e_{k}\mid \mathscr{F}_k}=0$.

\noindent (b) $\EXP{\| e_{i,k}\|^2\mid \mathscr{F}_k} \leq (b-1)\left(\nu^2+\|\nabla f_i(x_{i,k})\|^2\right)$.

\noindent (c) $\EXP{\|\bar e_{k}\|^2\mid \mathscr{F}_k} \leq (b-1)\nu^2 +\tfrac{b-1}{m}\|\bG(\bx_k)\|^2$.

\end{lemma}

\begin{proof}
\noindent (a) We can write
\begin{align*}
&\EXP{e_{i,k}\mid \mathscr{F}_k\cup\{\xi_{i,k}\}} = \nabla f_i(x_{i,k},\xi_{i,k})\\
&  - b\EXP{ \mathbf{U}_{\ell_{i,k}}\nabla^{\ell_{i,k}} f_i(x_{i,k},\xi_{i,k})\mid \mathscr{F}_k\cup\{\xi_{i,k}\}}\\
&= \nabla f_i(x_{i,k},\xi_{i,k}) - b\txsum_{\ell=1}^b b^{-1} \mathbf{U}_{\ell}\nabla^{\ell} f_i(x_{i,k},\xi_{i,k}) =0.
\end{align*}
The desired result follows by taking expectations from the preceding relation with respect to $\xi_{i,k}$.

\noindent (b) Throughout the proof, we use the compact notation $ {\tilde\nabla}_{i,k} \triangleq{\nabla} f_i(x_{i,k},\xi_{i,k})$.
Taking conditional expectations, we have
\begin{align*}
&\mathbb{E}\left[\left\|e_{i,k}\right\|^2\mid \mathscr{F}_k\cup\left(\cup_{j=1}^m\{\xi_{j,k}\}\right)\right] \\
&= \left\|{\tilde\nabla}_{i,k}\right\|^2 + b\textstyle\sum\nolimits_{\ell=1}^b\left\|\mathbf{U}_{\ell}{\tilde\nabla}^{\ell} _{i,k}\right\|^2-2{\tilde\nabla}_{i,k}^T\textstyle\sum\nolimits_{\ell=1}^b\mathbf{U}_{\ell}{\tilde\nabla}^{\ell} _{i,k}.
\end{align*}
We have $\textstyle\sum_{\ell=1}^b\left\|\mathbf{U}_{\ell}{\tilde\nabla}^{\ell} _{i,k}\right\|^2  \stackrel{(\ref{eqn:prop3_U_ell})}{=}  \left\|{\tilde\nabla}_{i,k}  \right\|^2$. From the two preceding relations, we obtain
\begin{align*}
&\mathbb{E}\left[\left\|e_{i,k}\right\|^2\mid \mathscr{F}_k\cup\left(\cup_{j=1}^m\{\xi_{j,k}\}\right)\right] = (b-1)\left\|{\tilde\nabla}_{i,k}\right\|^2.
\end{align*}
The desired relation holds by taking expectations with respect to $\cup_{j=1}^m \{\xi_{j,k}\}$ from both sides and invoking Assumption \ref{assum:stoch_errors}.

\noindent (c) This relation follows from part (b) and by noting that we have $\|\bar e_k\|^2\leq \tfrac{1}{m}\txsum_{i=1}^m\|e_{i,k}\|^2$. 
\end{proof}
The following two lemmas will be applied in the analysis and can be found in \cite{Pu_Nedic_20}.
\begin{lemma}\label{lem:Spectral_Norms}
Let Assumption \ref{assum:Stochastic_W} and \ref{assum:Graph}, hold. Let $\rho_W $, denote the spectral norm of $\mathbf{W}-\tfrac{1}{m}\mathbf{1}\mathbf{1}^T$ and  $\bar u \triangleq \tfrac{1}{m}\mathbf{1}^T\mathbf{u}$. Then, $\rho_W<1 $, and
$\|\mathbf{W}\mathbf{u}-\mathbf{1}\bar u\| \leq \rho_W \|\mathbf{u}-\mathbf{1}\bar u\| $ for all $\mathbf{u} \in \mathbb{R}^{m \times n}$.
\end{lemma}
\begin{lemma}\label{lem:bar_x_and_x_star}
Let Assumption \ref{assum:problem} hold. For any $\alpha \leq \tfrac{2}{\mu + L}$, we have $\|\bar x_k -\alpha \mathscr G(\bar x_k) -x^*\| \leq (1-\mu\alpha)\|\bar x_k- x^*\|$.
\end{lemma}
We also make use of the following result later in the analysis.
\begin{lemma}\label{lem:Norms}
Let $\mathbf{u}, \mathbf{v} \in \mathbb{R}^{m \times n}$. Then, the following holds.

\noindent (a) $\langle \mathbf{u}, \mathbf{v} \rangle  = \sum_{i=1}^m u_{i\bullet}v_{i\bullet}^T=\sum_{j=1}^n u_{\bullet j }^Tv_{\bullet j }. $ 

\noindent (b) $\|\mathbf{u}+ \mathbf{v}\|^2 = \|\mathbf{u} \|^2+2\langle \mathbf{u}, \mathbf{v} \rangle+ \| \mathbf{v}\|^2,$

where $\|\bullet\|$ denotes the Frobenius norm of a matrix. 

\noindent (c) $|\langle \mathbf{u}, \mathbf{v} \rangle| \leq \|\mathbf{u}\|\|\mathbf{v}\| \leq \tfrac{1}{2}\left(\lambda\|\mathbf{u}\|^2+\tfrac{1}{\lambda}\|\mathbf{v}\|^2\right)$ for any $\lambda>0$.
\end{lemma}
\section{Convergence and Rate Analysis}\label{sec:conv_anal}
Here, we present some properties of the gradient maps.
\begin{lemma}\label{lem:prelim_preperties}
Consider Algorithm \ref{algorithm:RB-DSGT}. Let Assumptions \ref{assum:problem}, \ref{assum:stoch_errors}, and \ref{assum:rand_block} hold. Then, for all $k\geq 0$, the following results hold.

\noindent (a) $b\bar y_k = G(\mathbf{x}_k,\boldsymbol{\xi}_k)-\bar{e}_{k}$. \quad  (b) $\EXP{b\bar y_k \mid \mathscr{F}_k}= G(\mathbf{x}_k)$.

\noindent (c) $\EXP{\|b\bar y_k - G(\mathbf{x}_k)\|^2 \mid \mathscr{F}_k} \leq  \left(\tfrac{1}{m}+b-1\right)\nu^2+\tfrac{b-1}{m}\|\bG(\bx_k)\|^2$.

\noindent (d) For any $\mathbf{u}, \mathbf{v} \in \Real{^{m\times n}}$, $\left\|G (\mathbf{u}) - G (\mathbf{v})\right\| \leq\tfrac{L}{\sqrt{m}}\|\mathbf{u}-\mathbf{v}\| $.

\noindent (e) $\left\|G(\mathbf{x}_k)-\sg(\bar{x}_k)\right\|\leq \tfrac{L}{\sqrt{m}}\|\mathbf{x}_k-\mathbf{1}\bar x_k\| $.

\noindent (f) $\left\|\sg(\bar{x}_k)\right\|\leq L\|\bar{x}_k-x^*\|$.

\noindent (g) $\|\mathbf{G}(\bx_k)\|^2 \leq   2L^2\|\mathbf{x}_k-\mathbf{1}\bar x_k\|^2 +2mL^2\|\bar x_k-x^*\|^2$.
\end{lemma}
\begin{proof}
\noindent(a) 
We use induction on $k$. For $k=0$ we have
\begin{align*}
b\bar{y}_{0} &= \tfrac{b}{m}\textstyle\sum\nolimits_{i=1}^m b^{-1}\left(\nabla f_i(x_{i,0},\xi_{i,0})-e_{i,0}\right)\\
&= {G}(\mathbf{x}_{0},{\bxi}_{0})-\bar{e}_{0}.
\end{align*}
Now suppose (a) holds for some $k$.  We have
\begin{align*}
b\bar{y}_{k+1} 
& = \tfrac{b}{m}\mathbf{1}^T\mathbf{y}_{k} + \tfrac{1}{m}\mathbf{1}^T\left(\mathbf{G}(\mathbf{x}_{k+1},{\bxi}_{k+1}) - \mathbf{e}_{k+1}\right) \\
&-  \tfrac{1}{m}\mathbf{1}^T\left(\mathbf{G}(\mathbf{x}_{k},{\bxi}_{k}) - \mathbf{e}_{k}\right)\\
& =b \bar{\mathbf{y}}_{k} +\tfrac{1}{m}\mathbf{1}^T\left(\mathbf{G}(\mathbf{x}_{k+1},{\bxi}_{k+1}) - \mathbf{e}_{k+1}\right) \\
& =\tfrac{1}{m}\mathbf{1}^T\left(\mathbf{G}(\mathbf{x}_{k+1},{\bxi}_{k+1}) - \mathbf{e}_{k+1}\right).
\end{align*}
Therefore, the induction hypothesis statement holds for $k+1$ and hence, the desired relation holds for all $k\geq 0$.

\noindent (b) Taking conditional expectations from the equation in part (a)   and utilizing Lemma \ref{lem:random_block_error}(a), we have
\begin{align*}
\EXP{b\bar y_k \mid \mathscr{F}_k} &= \EXP{G(\mathbf{x}_k,\boldsymbol{\xi}_k)-\bar{e}_{k}\mid \mathscr{F}_k} =G(\bx_k) -\EXP{\bar e_k \mid \sF_k }  \\ 
& = G(\bx_k)-\frm\txsum_{i=1}^m\EXP{e_{i,k} \mid \sF_k }=G(\bx_k).
\end{align*}

\noindent (c) From part (a), we have
\begin{align*}
&\EXP{\|b\bar y_k - G(\mathbf{x}_k)\|^2 \mid \mathscr{F}_k} = \mathbb{E}[\|G(\mathbf{x}_k,\boldsymbol{\xi}_k)-\bar{e}_{k} - G(\mathbf{x}_k)\|^2  \\
&  \mid \mathscr{F}_k] \leq  \tfrac{\nu^2}{m}+(b-1)\nu^2+\tfrac{b-1}{m}\|\bG(\bx_k)\|^2,
\end{align*}
where the last relation is obtained from Lemmas \ref{lem:random_block_error} and  \ref{assum:stoch_errors}.

\noindent (d) For any $\mathbf{u}, \mathbf{v} \in \Real{^{m\times n}}$, with $u_i,v_i \in \Real{^n}$ denoting the $i^{th}$ row of $\mathbf{u},\mathbf{v}$, respectively, we have
\begin{align*}
&\left\|G (\mathbf{u}) - G (\mathbf{v})\right\| = \left\|\tfrac{1}{m}\mathbf{1}^T\nabla \mathbf{f}(\mathbf{u}) - \tfrac{1}{m}\mathbf{1}^T \nabla \mathbf{f}\left(\mathbf{v}\right)\right\|\\
&  = \tfrac{1}{m}\left\|\textstyle\sum_{i=1}^m\nabla f_i(u_i)-\textstyle\sum_{i=1}^m \nabla f_i\left(v_i\right)\right\|
\leq\tfrac{L}{\sqrt{m}}\|\mathbf{u}-\mathbf{v}\|.
\end{align*}

\noindent (e)  We have 
$\left\|G(\mathbf{x}_k)-\sg(\bar{x}_k)\right\| = \left\|G(\mathbf{x}_k)- G(\mathbf{1}\bar{x}_k)\right\|
\leq \tfrac{L}{\sqrt{m}}\|\mathbf{x}_k-\mathbf{1}\bar x_k\|.$

\noindent (f) Invoking $ G(\mathbf{1}{x}^*) = \mathbf{0}$, we have
$\left\|\sg(\bar{x}_k)\right\| = \left\|G(\mathbf{1}\bar{x}_k)\right\|
= \left\|G(\mathbf{1}\bar{x}_k) - G(\mathbf{1}{x}^*)\right\| \leq L\|\bar{x}_k-x^*\|.$

\noindent (g)  We can write
\begin{align*}
&\|\mathbf{G}(\bx_k)\|^2 \leq 
 2L^2\|\mathbf{x}_k-\mathbf{1}\bar x_k\|^2 + 2\|\mathbf{G}(\mathbf{1}\bar x_k)-\mathbf{G}(\mathbf{1}  x^*)\|^2\\
& \leq 2L^2\|\mathbf{x}_k-\mathbf{1}\bar x_k\|^2 +2mL^2\|\bar x_k- x^*\|^2.
\end{align*}
\end{proof}
In the following, we derive three recursive error bounds that will be used later to derive the convergence rate statements.  
\begin{proposition}[Recursive error bounds]\label{prop:recursions} Consider Algorithm \ref{algorithm:RB-DSGT}.  Let Assumptions \ref{assum:problem}, \ref{assum:stoch_errors}, \ref{assum:rand_block}, \ref{assum:Stochastic_W}, and \ref{assum:Graph} hold.  Then, if $\gamma_k \leq \min\left\{ \tfrac{2b}{\mu + L}, \tfrac{b\mu}{4(b-1)L^2}\right\}$, for any $\eta>0$ we have 
\begingroup
\allowdisplaybreaks
\begin{align*}
(a)\ &\mathbb{E}\left[\|\bar x_{k+1}-x^*\|^2\right]\leq   (1-\tfrac{\mu b^{-1}\gamma_k}{2})\mathbb{E}\left[\|\bar x_k - x^*\|^2\right]\\
&+  \tfrac{b^{-1}\gamma_kL^2}{ m}\left(\tfrac{1}{\mu}+ b^{-1}(2b-1)\gamma_k\right)\mathbb{E}\left[\|\mathbf{x}_k-\mathbf{1} \bar x_k\|^2\right]\\
&+b^{-2}\gamma_k^2\left(\tfrac{1}{m}+b-1\right)\nu^2. \\
(b)\  &\mathbb{E}\left[\|\mathbf{x}_{k+1}-\mathbf{1} \bar x_{k+1}\|^2\right]\leq  \tfrac{1+\rho_W^2}{2} \mathbb{E}\left[ \|\mathbf{x}_k-\mathbf{1}{\bar x_k\|^2}\right]\\
&+\tfrac{\gamma_k^2(1+\rho_W^2)\rho_W^2}{1-\rho_W^2} \mathbb{E}\left[\|\mathbf{y}_k-\mathbf{1} \bar y_k	\|^2 \right]. \\
(c)\  &\mathbb{E}\left[ \| \mathbf{y}_{k+1}-\mathbf{1}\bar y_{k+1}\|^2\right]\leq \left( (1+b^{-1}\eta)\rho_W^2+\gamma_k^2(\tfrac{1}{b^2}+\tfrac{1}{b\eta})\right.\\
&\left.\times\left(2L^2\rho_W^2+\tfrac{2(b-1)L^2(1+\rho_W^2)\rho_W^2}{1-\rho_W^2}\right)\right)\mathbb{E}\left[\|\mathbf{y}_k-\mathbf{1} \bar y_k	\|^2 \right]\notag\\
&  +2L^2m(\tfrac{1}{b^2}+\tfrac{1}{b\eta})\left(b^{-2} L^2\gamma_k^2 \right.\notag\\
&\left.+(b-1)\left(3+L^2\gamma_k^2b^{-2} \right)\right)\mathbb{E}\left[\|\bar x_{k}-x^*\|^2 \right] \notag\\
&+2L^2\left(b^{-2}L^2\gamma_k^2+(b-1)(3+L^2\gamma_k^2b^{-2}  \right.\\
&\left.+b^{-1}\gamma_kL^2(\tfrac{1}{\mu}+b^{-1}(2b-1)\gamma_k ))\right.\\
& \left.+\|\mathbf{W}-\mathbf{I}\|^2\right)(\tfrac{1}{b^2}+\tfrac{1}{b\eta})\mathbb{E}\left[\|\mathbf{x}_k-\mathbf{1}\bar x_k\|^2\right]\notag\\
& +(\tfrac{1}{b^2}+\tfrac{1}{b\eta})\nu^2 \left(m L^2b^{-2}(\tfrac{1}{m}+b-1)\gamma_k^2+3mb\right).
\end{align*}
\endgroup
\end{proposition}
\begin{proof} (a) Multiplying both sides of first equation in \eqref{eqn:udpate_rules_compact} by $\tfrac{1}{m}\bunit^T$ and invoking $\bunit^T \mathbf{W} = \bunit^T$, we obtain $\bar{x}_{k+1} = \bar{x}_k - \gamma_k \bar{y}_k.$ Using Lemma \ref{lem:prelim_preperties}(b) and (c), we can write
\begin{align*}
&\mathbb{E}\left[\|\bar x_{k+1}-x^*\|^2\mid \mathscr{F}_k\right] = \mathbb{E}\left[\|\bar x_k - \gamma_k \bar y_k-x^*\|^2\mid \mathscr{F}_k\right] \\
& = { {\|\bar x_k - x^*\|^2}} -2 \gamma_k (\bar x_k -x^*)^T\mathbb{E}\left[\bar y_k\mid \mathscr{F}_k\right] +\gamma_k^2\mathbb{E}\left[\|\bar y_k\|^2\right.\\&\left.\mid \mathscr{F}_k\right]={ {\|\bar x_k - x^*\|^2}} -2 b^{-1}\gamma_k (\bar x_k -x^*)^TG(\mathbf{x_k})\\
&+b^{-2}\gamma_k^2\mathbb{E}\left[\|b\bar y_k-G(\mathbf{x_k})+G(\mathbf{x_k})\|^2\mid \mathscr{F}_k\right]\leq{ {\|\bar x_k - x^*\|^2}}\\
&  -2 b^{-1}\gamma_k (\bar x_k -x^*)^TG(\mathbf{x_k})+b^{-2}\gamma_k^2\|G(\mathbf{x_k})\|^2\\
&+b^{-2}\gamma_k^2\left(\left(\tfrac{1}{m}+b-1\right)\nu^2+\tfrac{b-1}{m}\|\bG(\bx_k)\|^2\right)\\
& = { {\|\bar x_k - x^*\|^2}} -2 b^{-1}\gamma_k (\bar x_k -x^*)^T(G(\mathbf{x_k})-\mathscr{G}(\bar x_k))\\
&-2 b^{-1}\gamma_k (\bar x_k -x^*)^T\mathscr{G}(\bar x_k) + b^{-2} \gamma_k^2\|G(\mathbf{x_k})-\mathscr{G}(\bar x_k)\|^2\\
&+b^{-2} \gamma_k^2\|\mathscr{G}(\bar x_k)\|^2+2b^{-2} \gamma_k^2(G(\mathbf{x_k})-\mathscr{G}(\bar x_k))^T\mathscr{G}(\bar x_k)\\
& +b^{-2}\gamma_k^2\left(\tfrac{1}{m}+b-1\right)\nu^2 + \tfrac{2L^2b^{-2} \gamma_k^2(b-1)}{m}\|\mathbf{x}_k-\mathbf{1}\bar x_k\|^2\\ 
& +2(b-1)L^2b^{-2} \gamma_k^2\|\bar x_k- x^*\|^2\\
& \leq \|\bar x_k - x^*-b^{-1}\gamma_k\mathscr{G}(\bar x_k)\|^2  \\
&-2 b^{-1}\gamma_k (\bar x_k -b^{-1}\gamma_k\mathscr{G}(\bar x_k)-x^*)^T(G(\mathbf{x_k})-\mathscr{G}(\bar x_k))\\ 
&+ b^{-2} \gamma_k^2\|G(\mathbf{x_k})-\mathscr{G}(\bar x_k)\|^2 +b^{-2}\gamma_k^2\left(\tfrac{1}{m}+b-1\right)\nu^2\\ 
& + \tfrac{2L^2b^{-2} \gamma_k^2(b-1)}{m}\|\mathbf{x}_k-\mathbf{1}\bar x_k\|^2 +2(b-1)L^2b^{-2} \gamma_k^2\|\bar x_k- x^*\|^2.
\end{align*}
Invoking Lemmas \ref{lem:bar_x_and_x_star} and \ref{lem:prelim_preperties}(e) we obtain
\begin{align*}
&{\mathbb{E}\left[\|\bar x_{k+1}-x^*\|^2\mid \mathscr{F}_k\right]} \leq (1-b^{-1}\mu\gamma_k)^2{ {\|\bar x_k - x^*\|^2}}  \\
&+ 2b^{-1}\gamma_k (1-b^{-1}\mu\gamma_k)\|\bar x_k - x^*\|\|G(\mathbf{x_k})-\mathscr{G}(\bar x_k)\|\\ 
 &+ \tfrac{b^{-2}\gamma_k^2L^2}{m}\|\mathbf{x}_k-\mathbf{1} \bar x_k\|^2+b^{-2}\gamma_k^2\left(\tfrac{1}{m}+b-1\right)\nu^2 \\
&+ \tfrac{2L^2b^{-2} \gamma_k^2(b-1)}{m}\|\mathbf{x}_k-\mathbf{1}\bar x_k\|^2 +2(b-1)L^2b^{-2} \gamma_k^2\|\bar x_k- x^*\|^2\\
&\leq \left((1-b^{-1}\mu\gamma_k)^2+2b^{-2}(b-1)\gamma_k^2L^2\right){ {\|\bar x_k - x^*\|^2}}\\
&+  \tfrac{2b^{-1}\gamma_k L(1-b^{-1}\mu\gamma_k)}{\sqrt{m}}\|\bar x_k - x^*\|\|\mathbf{x}_k-\mathbf{1} \bar x_k\|\notag\\
&+  \tfrac{b^{-2}(2b-1)\gamma_k^2L^2}{m}\|\mathbf{x}_k-\mathbf{1} \bar x_k\|^2+b^{-2}\gamma_k^2\left(\tfrac{1}{m}+b-1\right)\nu^2.
\end{align*}
Note that we can write
\begin{align*}
&\tfrac{2b^{-1}\gamma_kL (1-b^{-1}\mu\gamma_k)}{\sqrt{m}}\|\bar x_k - x^*\|\|\mathbf{x}_k-\mathbf{1} \bar x_k\|\notag\\
&= 2b^{-1}\gamma_k\left(\sqrt{\mu}(1-b^{-1}\mu\gamma_k)\|\bar x_k - x^*\|\right)\left(\tfrac{L\|\mathbf{x}_k-\mathbf{1} \bar x_k\|}{\sqrt{\mu m}}\right)\notag\\
& \leq b^{-1}\gamma_k\left(\mu(1-b^{-1}\mu\gamma_k)^2\|\bar x_k - x^*\|^2+\tfrac{L^2\|\mathbf{x}_k-\mathbf{1} \bar x_k\|^2}{\mu m}\right).
\end{align*}
From the preceding two relations, we obtain
\begin{align*}
&{\mathbb{E}\left[\|\bar x_{k+1}-x^*\|^2\mid \mathscr{F}_k\right]}\leq \left((1-b^{-1}\mu\gamma_k)^2(1+\mu b^{-1}\gamma_k)\right.\\
&\left.+2b^{-2}(b-1)\gamma_k^2L^2\right){\|\bar x_k - x^*\|^2}+  \tfrac{b^{-1}\gamma_kL^2}{ m}\left(b^{-1}(2b-1)\gamma_k \right.\\&+\left.\tfrac{1}{\mu}+b^{-2}\gamma_k^2\left(\tfrac{1}{m}+b-1\right)\nu^2 \right)\|\mathbf{x}_k-\mathbf{1} \bar x_k\|^2.
\end{align*}
From $  \gamma_k \leq \tfrac{b\mu}{4(b-1)L^2}$, we obtain 
\begin{align*}
&\mathbb{E}\left[\|\bar x_{k+1}-x^*\|^2\right]\leq   (1-\tfrac{\mu b^{-1}\gamma_k}{2})\mathbb{E}\left[\|\bar x_k - x^*\|^2\right]\\
&+  \tfrac{b^{-1}\gamma_kL^2}{ m}\left(\tfrac{1}{\mu}+ b^{-1}(2b-1)\gamma_k\right)\mathbb{E}\left[\|\mathbf{x}_k-\mathbf{1} \bar x_k\|^2\right]\\
&+b^{-2}\gamma_k^2\left(\tfrac{1}{m}+b-1\right)\nu^2.
\end{align*}

\noindent (b) From Equation \ref{eqn:udpate_rules_compact} and invoking Lemma \ref{lem:Norms}(b), we have
\begin{align*}
&{\|\mathbf{x}_{k+1}-\mathbf{1} \bar x_{k+1}\|^2} = \|\mathbf{W}\mathbf{x}_k-\gamma_k\mathbf{W}\mathbf{y}_k-\mathbf{1}(\bar x_k - \gamma_k\bar y_k)\|^2\\
&=\|\mathbf{W}\mathbf{x}_k-\mathbf{1}\bar x_k\|^2 -2\gamma_k\langle \mathbf{W}\mathbf{x}_k-\mathbf{1}\bar x_k,\mathbf{W}\mathbf{y}_k-\mathbf{1}\bar y_k\rangle\\
&+\gamma_k^2\|\mathbf{W}\mathbf{y}_k-\mathbf{1}\bar y_k\|^2.
\end{align*}
By Invoking Lemma \ref{lem:Spectral_Norms} and Lemma \ref{lem:Norms}(c), we obtain
\begin{align*}
&{\|\mathbf{x}_{k+1}-\mathbf{1} \bar x_{k+1}\|^2}\\
&=\rho_W^2{ {\|\mathbf{x}_k-\mathbf{1}\bar x_k\|^2}} +2\gamma_k\| \mathbf{W}\mathbf{x}_k-\mathbf{1}\bar x_k\|\|\mathbf{W}\mathbf{y}_k-\mathbf{1}\bar y_k\|\\ &+\rho_W^2\gamma_k^2{ {\|\mathbf{y}_k-\mathbf{1}\bar y_k\|^2}}\\
&\leq\rho_W^2{ {\|\mathbf{x}_k-\mathbf{1}\bar x_k\|^2}} +2\rho_W^2\gamma_k\|  \mathbf{x}_k-\mathbf{1}\bar x_k\|\| \mathbf{y}_k-\mathbf{1}\bar y_k\|\\
&+\rho_W^2\gamma_k^2{ {\|\mathbf{y}_k-\mathbf{1}\bar y_k\|^2}}\\
&\leq\rho_W^2{ {\|\mathbf{x}_k-\mathbf{1}\bar x_k\|^2}} +\rho_W^2\gamma_k\left(\tfrac{1-\rho_W^2}{2\gamma_k\rho_W^2}   \|\mathbf{x}_k-\mathbf{1}\bar x_k\|^2\right.\\
&\left.+\tfrac{2\gamma_k\rho_W^2}{1-\rho_W^2}\| \mathbf{y}_k-\mathbf{1}\bar y_k\|^2\right) +\rho_W^2\gamma_k^2{ {\|\mathbf{y}_k-\mathbf{1}\bar y_k\|^2}}\\
&= \tfrac{1+\rho_W^2}{2}  { \|\mathbf{x}_k-\mathbf{1}{\bar x_k\|^2}}+\tfrac{\gamma_k^2(1+\rho_W^2)\rho_W^2}{1-\rho_W^2}{ {\| \mathbf{y}_k-\mathbf{1}\bar y_k\|^2}}. 
\end{align*}
Taking expectations from the last relation, we obtain (b). 

\noindent (c) Next we obtain the third recursive relation. For the ease of presentation, we will use the following compact notation. 
\begin{align*}
&\mathbf{G}_k \triangleq \mathbf{G}(\mathbf{x}_{k}), \quad \tilde{\mathbf{G}}_k \triangleq \mathbf{G}(\mathbf{x}_{k},\boldsymbol{\xi}_{k}),\quad \tilde{\mathbf{G}}^e_k \triangleq\tilde{\mathbf{G}}_k -\mathbf{e}_k,\\
&\nabla_{i,k}\triangleq \nabla f_i(x_{i,k}),\quad \tilde \nabla_{i,k} \triangleq \nabla f_i(x_{i,k},\xi_{i,k}),\\
&\tilde \nabla^e_{i,k} \triangleq \nabla f_i(x_{i,k},\xi_{i,k})-e_{i,k}.
\end{align*}
Note that $\mathbb{E}\left[\tilde{\mathbf{G}}^e_k\mid \mathscr{F}_k\right] = \mathbf{G}_k,\quad \mathbb{E}\left[\tilde{\mathbf{G}}^e_{k+1}\mid \mathscr{F}_{k+1}\right] = \mathbf{G}_{k+1}$.
From Algorithm \ref{algorithm:RB-DSGT},  we have, $ \| \mathbf{y}_{k+1}-\mathbf{1}\bar y_{k+1}\|^2$
\begin{align}\label{eqn:rec3_first_ineq}
& \leq \|\mathbf{W}\mathbf{y}_k+b^{-1}\tilde{\mathbf{G}}^e_{k+1}-b^{-1}\tilde{\mathbf{G}}^e_k-\mathbf{1}\bar y_{k}+\mathbf{1}\bar y_{k}-\mathbf{1}\bar y_{k+1}\|^2\notag\\
& = \|\mathbf{W}\mathbf{y}_k-\mathbf{1}\bar y_{k}\|^2+ b^{-2}\|\tilde{\mathbf{G}}^e_{k+1}-\tilde{\mathbf{G}}^e_k\|^2\notag\\
&+2b^{-1}\langle \mathbf{W} \mathbf{y}_k -\mathbf{1} \bar y_k,\tilde{\mathbf{G}}^e_{k+1}-\tilde{\mathbf{G}}^e_k \rangle + m\|\bar y_{k}-\bar y_{k+1}\|^2 \notag\\
& +2\langle  \mathbf{y}_{k+1} -\mathbf{1} \bar y_k,\mathbf{1}(\bar y_{k}-\bar y_{k+1}) \rangle \leq \rho_W^2\|\mathbf{y}_k-\mathbf{1}\bar y_{k}\|^2\notag\\ &+b^{-2}\|\tilde{\mathbf{G}}^e_{k+1}-\tilde{\mathbf{G}}^e_k\|^2\notag+ 2b^{-1}\langle \mathbf{W} \mathbf{y}_k -\mathbf{1} \bar y_k,\tilde{\mathbf{G}}^e_{k+1}-\tilde{\mathbf{G}}^e_k \rangle\notag\\&-m\|\bar y_{k}-\bar y_{k+1}\|^2\leq\notag (1+b^{-1}\eta)\rho_W^2{{\| \mathbf{y}_k-\mathbf{1}\bar y_k\|^2}}\\&+ (\tfrac{1}{b^2}+\tfrac{1}{b\eta})\|\tilde{\mathbf{G}}^e_{k+1}-\tilde{\mathbf{G}}^e_k\|^2,
\end{align}
where $\eta>0$ is an arbitrary scalar. In the following, we present a few intermediary results that will be used to derive the third recursive inequality.

\noindent \textbf{Claim 1}: The following holds
\begin{align}\label{eqn:third_recursion-claim1}
& \mathbb{E}\left[\|\tilde{\mathbf{G}}^e_{k+1}-\tilde{\mathbf{G}}^e_k\|^2\mid \mathscr{F}_k\right] \leq  \mathbb{E}\left[\|\mathbf{G}_{k+1}-\mathbf{G}_k\|^2\mid \mathscr{F}_k\right] \nonumber\\
&+ \mathbb{E}\left[\| \mathbf{G}_{k+1}\|^2 \mid \mathscr{F}_k\right]  + \mathbb{E}\left[\| \tilde{\mathbf{G}}^e_{k+1} -\mathbf{G}_{k+1} \|^2 \mid \mathscr{F}_k\right] \notag\\
&+ 2\mathbb{E}\left[\| \tilde{\mathbf{G}}^e_k-\mathbf{G}_k\|^2 \mid \mathscr{F}_k\right].
\end{align}
\begin{proof}We can write
\begin{align}\label{eqn:third_recursion_proof_eqn1}
& \mathbb{E}\left[\|\tilde{\mathbf{G}}^e_{k+1}-\tilde{\mathbf{G}}^e_k\|^2\mid \mathscr{F}_k\right] = \mathbb{E}\left[\|\mathbf{G}_{k+1}-\mathbf{G}_k\|^2\mid \mathscr{F}_k\right] \notag\\
& + 2\mathbb{E}\left[\langle \mathbf{G}_{k+1}, \tilde{\mathbf{G}}^e_{k+1}-\tilde{\mathbf{G}}^e_k-\mathbf{G}_{k+1}+\mathbf{G}_k\rangle \mid \mathscr{F}_k\right]\notag \\
& -2\mathbb{E}\left[\langle \mathbf{G}_k, \tilde{\mathbf{G}}^e_{k+1}-\tilde{\mathbf{G}}^e_k-\mathbf{G}_{k+1}+\mathbf{G}_k\rangle \mid \mathscr{F}_k\right] \notag\\
& +\mathbb{E}\left[\| \tilde{\mathbf{G}}^e_{k+1}-\tilde{\mathbf{G}}^e_k-\mathbf{G}_{k+1}+\mathbf{G}_k\|^2 \mid \mathscr{F}_k\right].
\end{align}
Note that since $\mathbf{x}_{k+1}$ is characterized in terms of $\boldsymbol{\xi}_k$, using Assumption \ref{assum:stoch_errors},  Assumption \ref{assum:rand_block}, and Lemma \ref{lem:random_block_error} we have
$\mathbb{E}\left[\tilde{\mathbf{G}}^e_{k+1} -\mathbf{G}_{k+1}\mid \mathscr{F}_k\right]=0,$ and
 $\mathbb{E}\left[\tilde{\mathbf{G}}^e_{k} -\mathbf{G}_{k}\mid \mathscr{F}_k\right]  =0$. Thus, we obtain
\begin{align*}
&\mathbb{E}\left[\langle \mathbf{G}_k, \tilde{\mathbf{G}}^e_{k+1}-\tilde{\mathbf{G}}^e_k-\mathbf{G}_{k+1}+\mathbf{G}_k\rangle \mid \mathscr{F}_k\right] 
= 0.
 \end{align*}
We can also write 
\begin{align*}
& \mathbb{E}\left[\langle \mathbf{G}_{k+1}, \tilde{\mathbf{G}}^e_{k+1}-\tilde{\mathbf{G}}^e_k-\mathbf{G}_{k+1}+\mathbf{G}_k\rangle \mid \mathscr{F}_k\right]\\
&= \mathbb{E}_{\boldsymbol{\xi}_k,\boldsymbol{\ell}_k}\left[\langle \mathbf{G}_{k+1}, \mathbb{E}_{\boldsymbol{\xi}_{k+1},\boldsymbol{\ell}_{k+1}}\left[\tilde{\mathbf{G}}^e_{k+1} -\mathbf{G}_{k+1} \mid \mathscr{F}_{k+1} \right]\rangle\right]+ \\
& \mathbb{E}\left[\langle \mathbf{G}_{k+1}, -\tilde{\mathbf{G}}^e_k +\mathbf{G}_k\rangle \mid \mathscr{F}_k\right]= \mathbb{E}\left[\langle \mathbf{G}_{k+1}, -\tilde{\mathbf{G}}^e_k +\mathbf{G}_k\rangle \mid \mathscr{F}_k\right].
\end{align*}
From the preceding relations, we have
\begin{align*}
& \mathbb{E}\left[\|\tilde{\mathbf{G}}^e_{k+1}-\tilde{\mathbf{G}}^e_k\|^2\mid \mathscr{F}_k\right] 
 \leq  \mathbb{E}\left[\|\mathbf{G}_{k+1}-\mathbf{G}_k\|^2\mid \mathscr{F}_k\right] \\
&+ 2\mathbb{E}\left[\langle \mathbf{G}_{k+1}, -\tilde{\mathbf{G}}^e_k +\mathbf{G}_k\rangle \mid \mathscr{F}_k\right] \\
& + \mathbb{E}\left[\| \tilde{\mathbf{G}}^e_{k+1} -\mathbf{G}_{k+1} \|^2 \mid \mathscr{F}_k\right]+ \mathbb{E}\left[\| \tilde{\mathbf{G}}^e_k-\mathbf{G}_k\|^2 \mid \mathscr{F}_k\right]\\
&-2\mathbb{E}\left[\langle \tilde{\mathbf{G}}^e_{k+1} -\mathbf{G}_{k+1},\tilde{\mathbf{G}}^e_k-\mathbf{G}_k \rangle \mid \mathscr{F}_k\right].
\end{align*}
Note that we have
\begin{align*}
& \mathbb{E}\left[\langle \tilde{\mathbf{G}}^e_{k+1} -\mathbf{G}_{k+1},\tilde{\mathbf{G}}^e_k-\mathbf{G}_k \rangle \mid \mathscr{F}_k\right] 
= \mathbb{E}_{\boldsymbol{\xi}_k,\boldsymbol{\ell}_k}\left[\langle \mathbb{E}_{\boldsymbol{\xi}_{k+1},\boldsymbol{\ell}_{k+1}}\right.\\ &\left.\left[\tilde{\mathbf{G}}^e_{k+1} -\mathbf{G}_{k+1}\mid \mathscr{F}_{k+1} \right],\tilde{\mathbf{G}}^e_k-\mathbf{G}_k\rangle \right]=0.
\end{align*}
Also, we have,
$2\mathbb{E}\left[\langle \mathbf{G}_{k+1}, -\tilde{\mathbf{G}}^e_k +\mathbf{G}_k\rangle \mid \mathscr{F}_k\right] $
\begin{align*}
&\leq \mathbb{E}\left[\| \mathbf{G}_{k+1}\|^2 \mid \mathscr{F}_k\right] +\mathbb{E}\left[\| \tilde{\mathbf{G}}^e_{k} -\mathbf{G}_{k} \|^2 \mid \mathscr{F}_k\right].
\end{align*}
From the last three relations, we obtain Claim 1.
\end{proof}

\noindent \textbf{Claim 2:} We have, \ $\mathbb{E}\left[\| \tilde{\mathbf{G}}^e_{k+1} -\mathbf{G}_{k+1} \|^2 \mid \mathscr{F}_k\right]$
\begin{align}\label{eqn:third_recursion-claim2}
& \leq mb\nu^2+(b-1)\mathbb{E}\left[\|\mathbf{G}_{k+1}\|^2\mid \mathscr{F}_{k}\right],\notag\\
&\mathbb{E}\left[\| \tilde{\mathbf{G}}^e_k-\mathbf{G}_k\|^2 \mid \mathscr{F}_k\right] \leq  mb\nu^2 + (b-1)\|\bG_k\|^2.
\end{align}

\begin{proof} From Assumption \ref{assum:stoch_errors} and Lemma \ref{lem:random_block_error}, we have
\begin{align*}
&\mathbb{E}\left[\| \tilde{\mathbf{G}}^e_k-\mathbf{G}_k\|^2 \mid \mathscr{F}_k\right] =\mathbb{E}\left[\| \tilde{\mathbf{G}}_k -\mathbf{e}_k-\mathbf{G}_k\|^2 \mid \mathscr{F}_k\right] \\
& = \mathbb{E}\left[\| \tilde{\mathbf{G}}_k -\mathbf{G}_k\|^2 \mid \mathscr{F}_k\right]+\mathbb{E}\left[\| \mathbf{e}_k\|^2 \mid \mathscr{F}_k\right]\\
& = mb\nu^2 + (b-1)\|\bG_k\|^2.
\end{align*}
Using this relation, we can also write
\begin{align*}
&\mathbb{E}\left[\| \tilde{\mathbf{G}}^e_{k+1} -\mathbf{G}_{k+1} \|^2 \mid \mathscr{F}_k\right] \\
&=\mathbb{E}_{\boldsymbol{\xi}_k,\boldsymbol{\ell}_k}\left[\mathbb{E}_{\boldsymbol{\xi}_{k+1},\boldsymbol{\ell}_{k+1}}\left[\|\tilde{\mathbf{G}}^e_{k+1} -\mathbf{G}_{k+1} \|^2  \mid \mathscr{F}_{k+1} \right]\right]\\
&\leq mb\nu^2+(b-1)\mathbb{E}\left[\|\mathbf{G}_{k+1}\|^2\mid \mathscr{F}_{k}\right].
\end{align*}
\end{proof}
\noindent \textbf{Claim 3}: The following inequality holds.
\begin{align}\label{eqn:third_recursion-claim3}
&\mathbb{E}\left[\|\mathbf{G}_{k+1}-\mathbf{G}_k\|^2 \mid \mathscr{F}_k\right]\leq 2L^2\left(b^{-2}L^2\gamma_k^2+\|\mathbf{W}-\mathbf{I}\|^2\right)\notag\\
&\|\mathbf{x}_k-\mathbf{1}\bar x_k\|^2+2L^2\rho_W^2\gamma_k^2\mathbb{E}\left[\|\mathbf{y}_k-\mathbf{1} \bar y_k	\|^2 \mid \mathscr{F}_k\right]\notag\\
&  +2b^{-2}m L^4\gamma_k^2\|\bar x_{k}-x^*\|^2 + L^2\gamma_k^2b^{-2}(b-1)\|\bG_k\|^2\notag\\
& + m L^2b^{-2}(\tfrac{1}{m}+b-1)\nu^2\gamma_k^2.
\end{align}
\begin{proof}From the Lipschitzian property of the local objectives we have 
$ \|\mathbf{G}_{k+1}-\mathbf{G}_k\|^2 \leq L^2\|\mathbf{x}_{k+1}-\mathbf{x}_k\|^2. $
Next, we estimate the term $\|\mathbf{x}_{k+1}-\mathbf{x}_k\|^2$. We have
\begin{align*}
&\|\mathbf{x}_{k+1}-\mathbf{x}_{k}\|^2 =\|\mathbf{W}\mathbf{x}_k-\gamma_k\mathbf{W}\mathbf{y}_k-\mathbf{x}_k\|^2\\
& = \|(\mathbf{W}-\mathbf{I})(\mathbf{x}_k-\mathbf{1}\bar x_k)-\gamma_k\mathbf{W}\mathbf{y}_k\|^2\\
& =  \|\mathbf{W}-\mathbf{I}\|^2\|\mathbf{x}_k-\mathbf{1}\bar x_k\|^2+\gamma_k^2\|\mathbf{W}\mathbf{y}_k-\mathbf{1} \bar y_k	\|^2\\
&+m\gamma_k^2\|\bar y_k\|^2 -2\gamma_k\langle (\mathbf{W}-\mathbf{I})(\mathbf{x}_k-\mathbf{1}\bar x_k), \mathbf{W}\mathbf{y}_k-\mathbf{1}\bar y_k\rangle \\
& \leq \|\mathbf{W}-\mathbf{I}\|^2\|\mathbf{x}_k-\mathbf{1}\bar x_k\|^2+\rho_W^2\gamma_k^2\|\mathbf{y}_k-\mathbf{1} \bar y_k\|^2\\
&+m\gamma_k^2\|\bar y_k\|^2 +2\rho_{W}\gamma_k\|\mathbf{W}-\mathbf{I}\|\|\mathbf{x}_k-\mathbf{1}\bar x_k\|  \|\mathbf{y}_k-\mathbf{1}\bar y_k \|\\
& \leq 2\|\mathbf{W}-\mathbf{I}\|^2\|\mathbf{x}_k-\mathbf{1}\bar x_k\|^2+2\rho_W^2\gamma_k^2\|\mathbf{y}_k-\mathbf{1} \bar y_k	\|^2\\
&+m\gamma_k^2b^{-2}\|b\bar y_k\|^2.
\end{align*}

From Lemma \ref{lem:prelim_preperties}(b) and (c) we have
\begin{align*}
&\mathbb{E}\left[\|b\bar y_k\|^2\mid \mathscr{F}_k\right] \leq  \left(\tfrac{1}{m}+b-1\right)\nu^2+\tfrac{b-1}{m}\|\bG_k\|^2+ \left\|G_k\right\|^2.
\end{align*}
Also, from Lemma \ref{lem:prelim_preperties}(e) and (f) we have
\begin{align*}
& \left\|G_k\right\|^2 \leq2\|G_k - \mathscr{G}(\bar{x}_k) \|^2+ 2\|\mathscr{G}(\bar{x}_k) \|^2\\
& \leq \tfrac{2L^2}{m}\left\|\mathbf{x}_k-\mathbf{1}\bar x_k\right\|^2+ 2L^2\|\bar x_k-x^*\|^2.
\end{align*}
From the preceding relations, we obtain Claim 3. 
\end{proof}

\noindent \textbf{Claim 4}: We have
\begin{align*}
& \mathbb{E}\left[\|\mathbf{G}_{k+1}\|^2\mid \mathscr{F}_{k}\right] \leq 2L^2m\|\bar x_k - x^*\|^2\notag\\
&+\tfrac{2L^2\gamma_k^2(1+\rho_W^2)\rho_W^2}{1-\rho_W^2} \mathbb{E}\left[\|\mathbf{y}_k-\mathbf{1}\bar y_k\|^2\mid \mathscr{F}_{k}\right]\notag\\
&+2L^2\left(1 +b^{-1}\gamma_kL^2(\tfrac{1}{\mu}+b^{-1}(2b-1)\gamma_k )\right){ \|\mathbf{x}_k-\mathbf{1}{\bar x_k\|^2}}\notag\\
&+ 2mL^2b^{-2}\nu^2(\tfrac{1}{m}+b-1) \gamma_k^2.
\end{align*}
\begin{proof}
From Lemma \ref{lem:prelim_preperties}(g) we can write
\begin{align*}
&\mathbb{E}\left[\|\mathbf{G}_{k+1}\|^2\mid \mathscr{F}_{k}\right]\leq 2L^2\mathbb{E}\left[\|\mathbf{x}_{k+1}-\mathbf{1}\bar x_{k+1}\|^2\mid \mathscr{F}_{k}\right]\\
&+2mL^2\mathbb{E}\left[\|\bar x_{k+1}- x^*\|^2\mid \mathscr{F}_{k}\right].
\end{align*} 
From the first two recursive bounds, substituting  $\mathbb{E}\left[\|\mathbf{x}_{k+1}-\mathbf{1}\bar x_{k+1}\|^2\mid \mathscr{F}_{k}\right]$ and $\mathbb{E}\left[\|\bar x_{k+1}- x^*\|^2\mid \mathscr{F}_{k}\right]$,  we can conclude Claim 4.
\end{proof}
By combining \eqref{eqn:rec3_first_ineq}, Claims 1 to 4, and Lemma \ref{lem:prelim_preperties}(g), we obtain
\begin{align*}
 &\mathbb{E}\left[ \| \mathbf{y}_{k+1}-\mathbf{1}\bar y_{k+1}\|^2\right]\notag\\
&\leq \left( (1+b^{-1}\eta)\rho_W^2+\gamma_k^2(\tfrac{1}{b^2}+\tfrac{1}{b\eta})\left(2L^2\rho_W^2\right.\right. \notag\\
&\left.\left.+\tfrac{2(b-1)L^2(1+\rho_W^2)\rho_W^2}{1-\rho_W^2}\right)\right)\mathbb{E}\left[\|\mathbf{y}_k-\mathbf{1} \bar y_k	\|^2 \right]\notag\\
&  +2L^2m(\tfrac{1}{b^2}+\tfrac{1}{b\eta})\left(b^{-2} L^2\gamma_k^2 \right.\notag\\
&\left.+(b-1)\left(3+L^2\gamma_k^2b^{-2} \right)\right)\mathbb{E}\left[\|\bar x_{k}-x^*\|^2 \right] \notag\\
&+2L^2\left(b^{-2}L^2\gamma_k^2+(b-1)(3+L^2\gamma_k^2b^{-2}  \right.\\
&\left.+b^{-1}\gamma_kL^2(\tfrac{1}{\mu}+b^{-1}(2b-1)\gamma_k ))\right.\\
& \left.+\|\mathbf{W}-\mathbf{I}\|^2\right)(\tfrac{1}{b^2}+\tfrac{1}{b\eta})\mathbb{E}\left[\|\mathbf{x}_k-\mathbf{1}\bar x_k\|^2\right]\notag\\
& +(\tfrac{1}{b^2}+\tfrac{1}{b\eta})\nu^2 \left(m L^2b^{-2}(\tfrac{1}{m}+b-1)\gamma_k^2+3mb\right).
\end{align*}
\end{proof}
\begin{table*}[t]
\renewcommand\thetable{1}
	\captionsetup{labelformat=empty}
\centering{
 \begin{tabular}{l | c  c  c }
   &(MNIST, $m=5$)&(Synthetic, $m=5$)&(Synthetic, $m=10$) \\
 \hline\\
\rotatebox[origin=c]{90}{{\footnotesize $ \log \left(\sum_{i=1}^mf_i(\bar x_k)\right)$}}
&
\begin{minipage}{.22\textwidth}
\includegraphics[scale=0.28, angle=0]{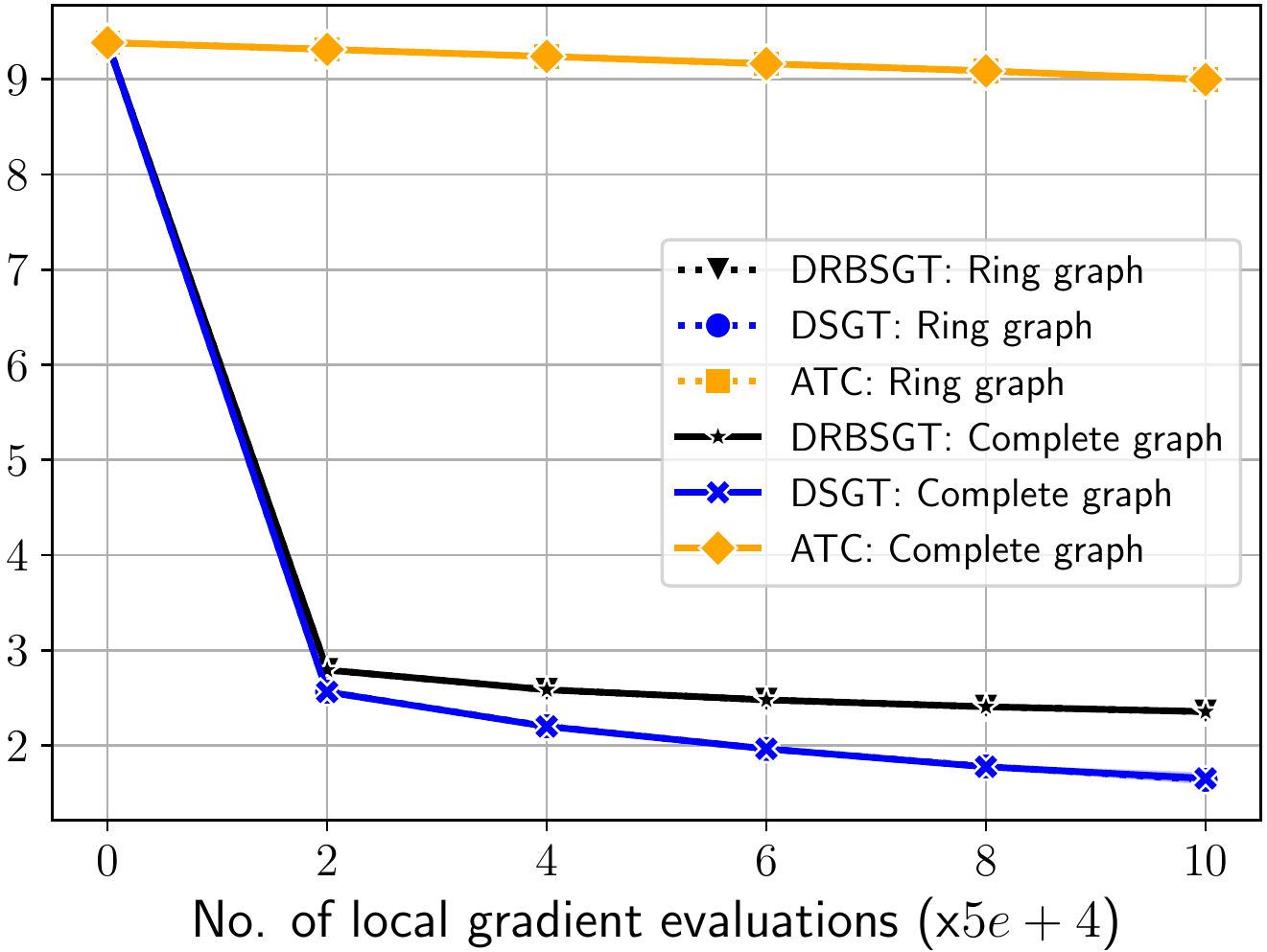}
\end{minipage}
&
\begin{minipage}{.22\textwidth}
\includegraphics[scale=0.28, angle=0]{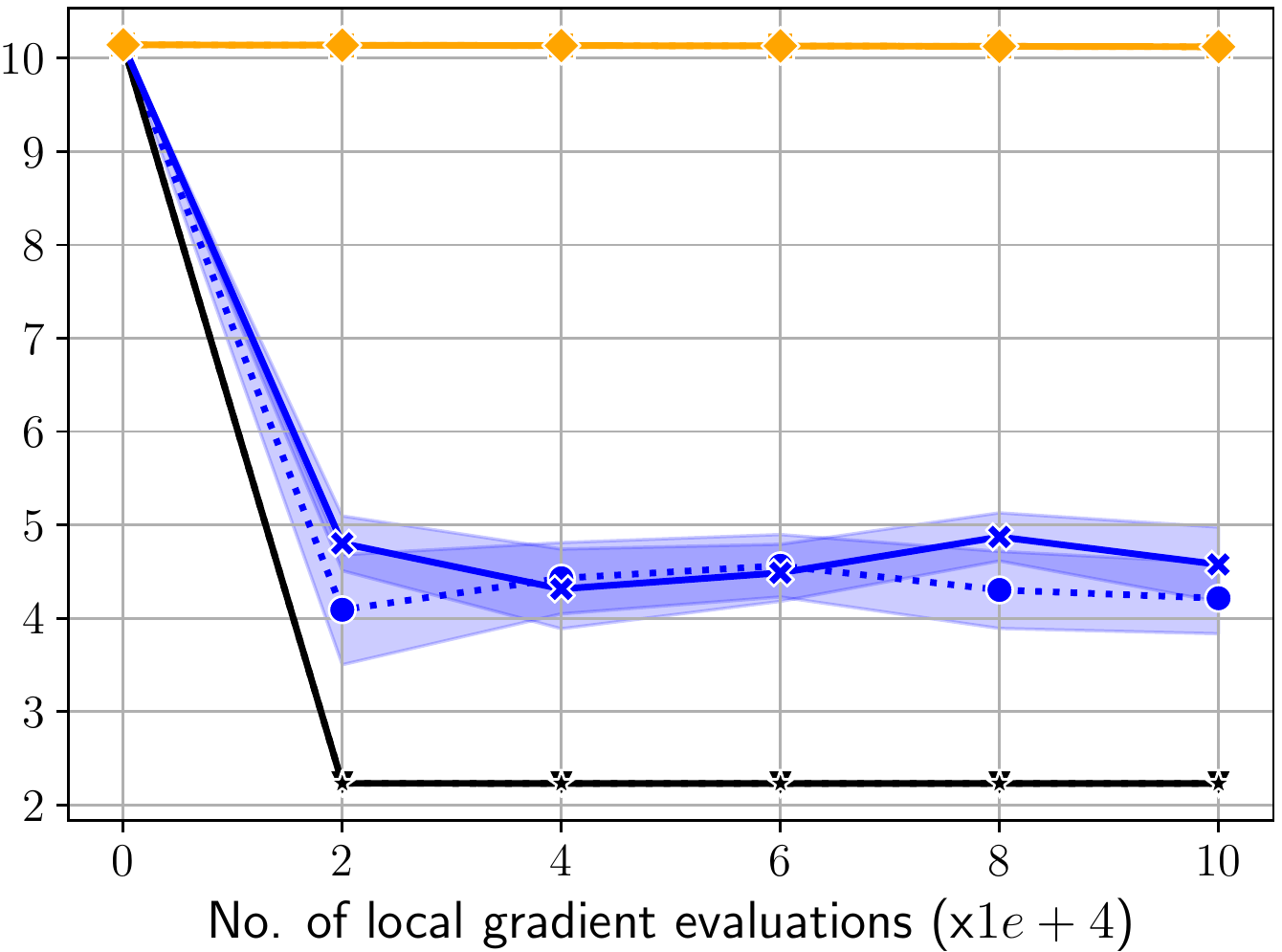}
\end{minipage}
	&
\begin{minipage}{.22\textwidth}
\includegraphics[scale=0.28, angle=0]{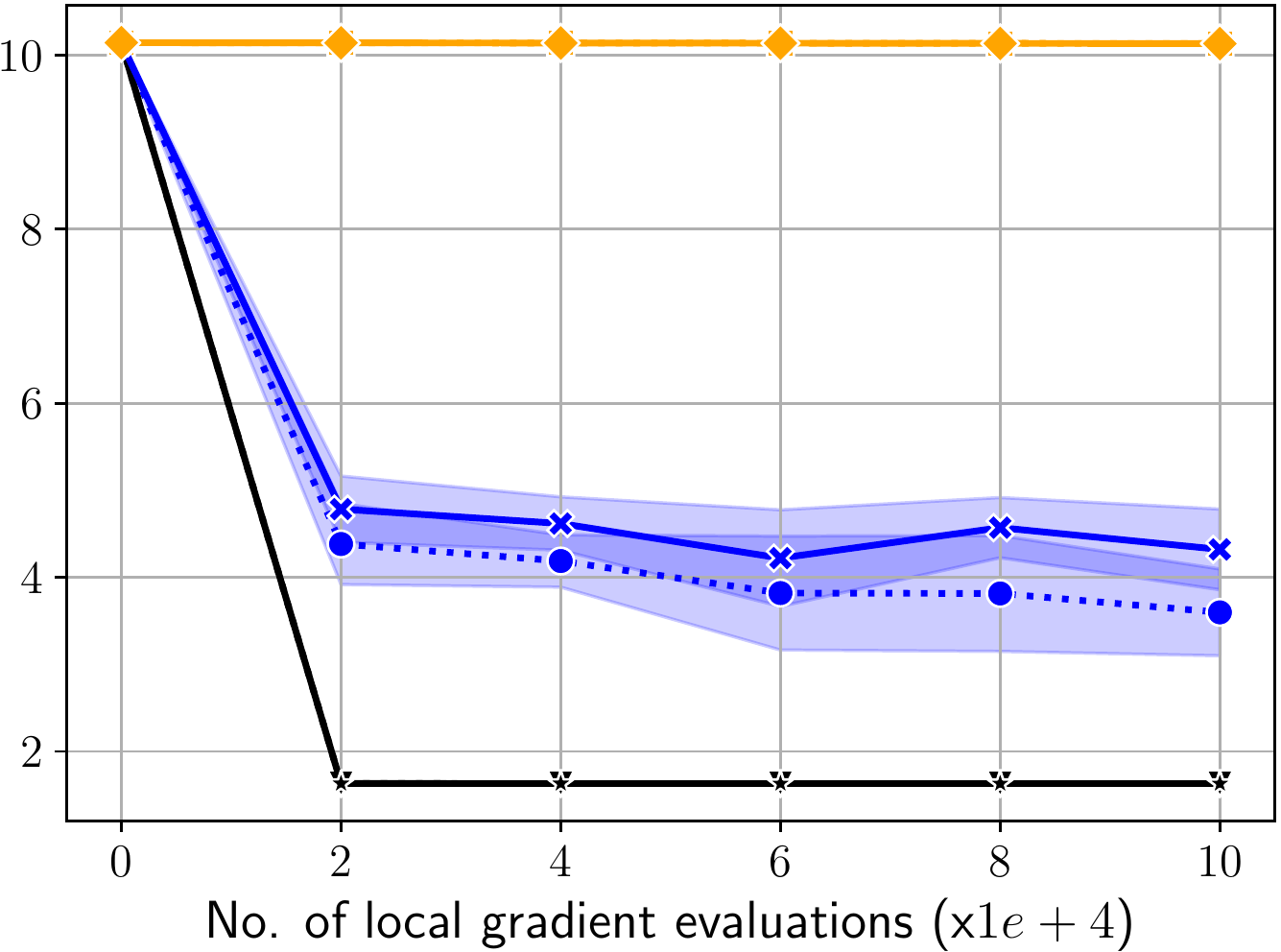}
\end{minipage}

\\
\hbox{}& & & \\
 \hline\\
\rotatebox[origin=c]{90}{{\footnotesize$ \log \left(\left\|\mathbf{x}_k-\mathbf{1}\bar{x}_k\right\|\right)$}}
&
\begin{minipage}{.22\textwidth}
\includegraphics[scale=0.28, angle=0]{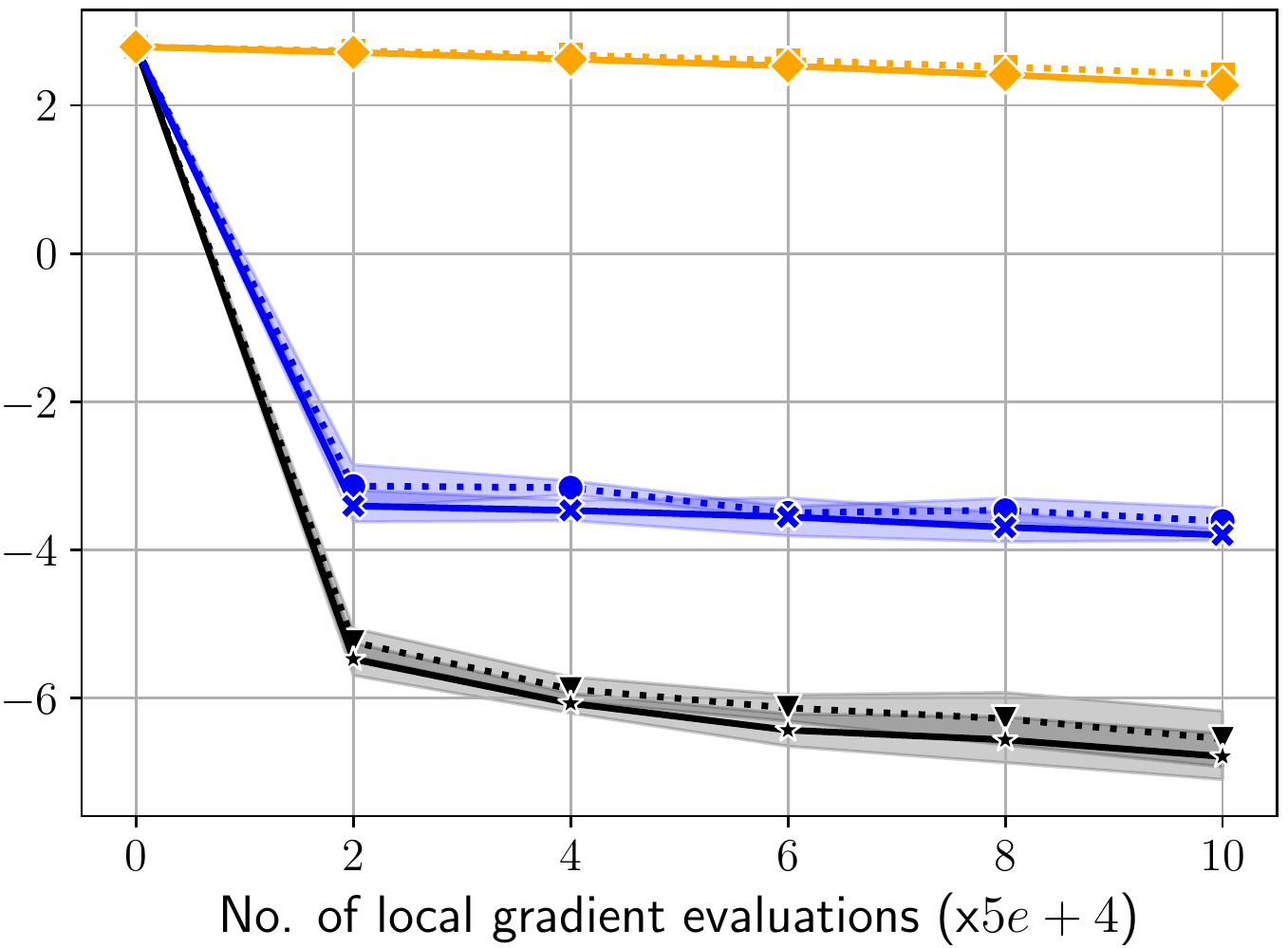}
\end{minipage}
&
\begin{minipage}{.22\textwidth}
\includegraphics[scale=0.28, angle=0]{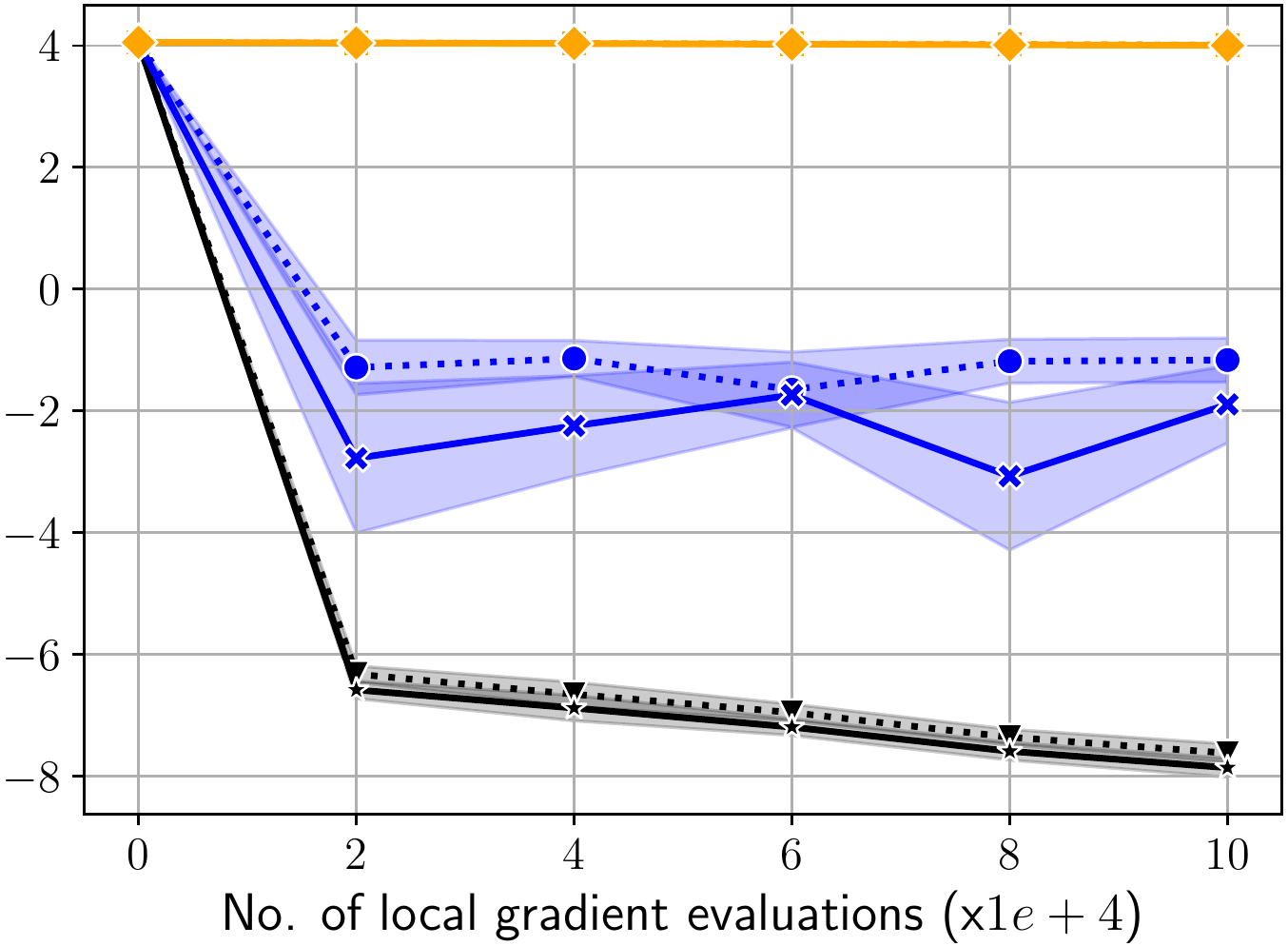}
\end{minipage}
&
\begin{minipage}{.22\textwidth}
\includegraphics[scale=0.28, angle=0]{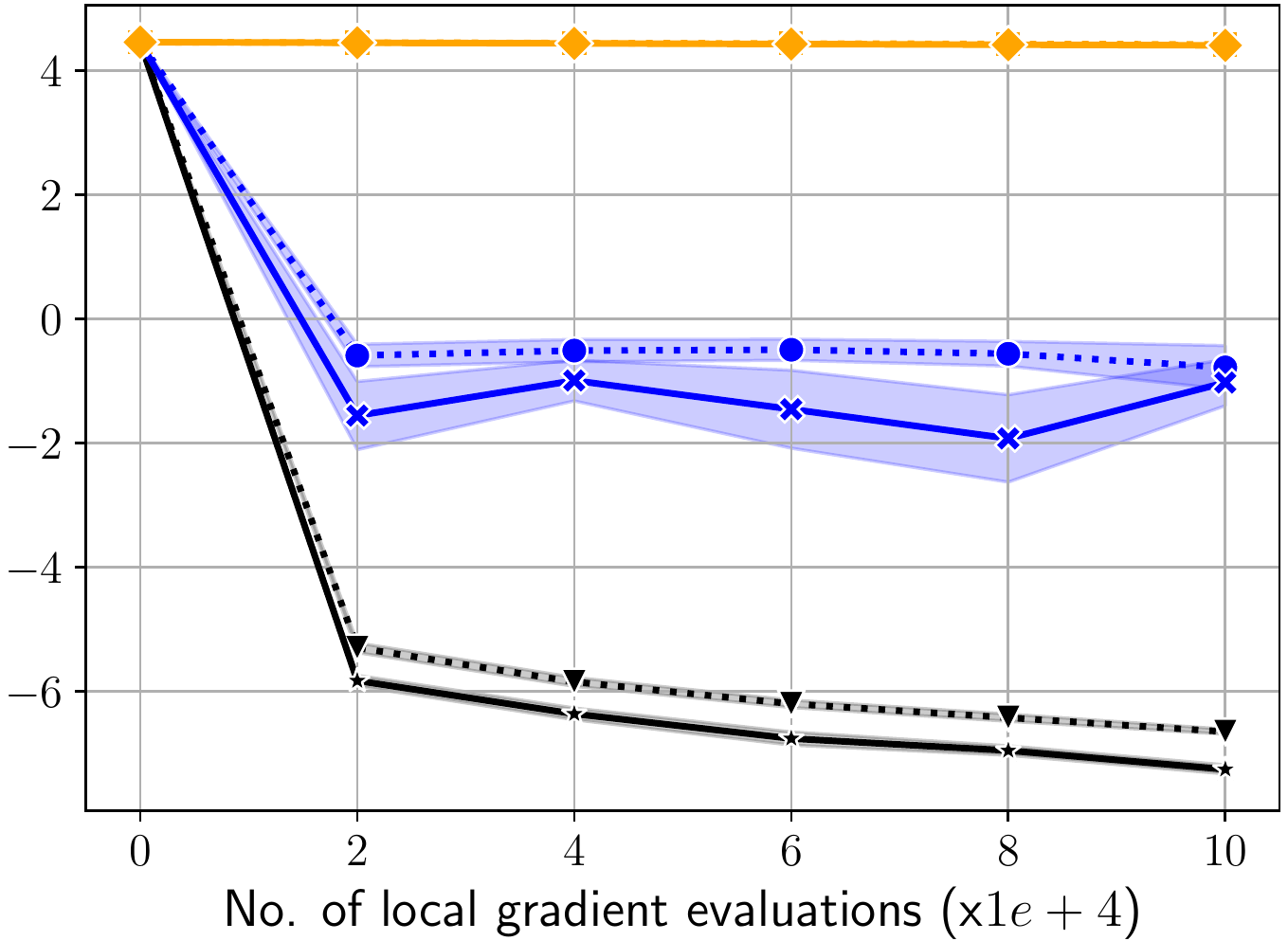}
\end{minipage}
\end{tabular}}
\captionof{figure}{{Figure 1}: Algorithm \ref{algorithm:RB-DSGT} vs. DSGT vs. ATC in terms of objective function value and consensus error}
\label{fig:Alg_Output}
\vspace{-.2in}
\end{table*}
We are now ready to present the main result of the paper. 
\begin{theorem}[Rate statements]\label{thm:rate} 
Consider Algorithm \ref{algorithm:RB-DSGT}.  Let Assumptions \ref{assum:problem}, \ref{assum:stoch_errors}, \ref{assum:rand_block}, \ref{assum:Stochastic_W}, and \ref{assum:Graph} hold. Let us define $\mbox{err}_{1,k} \triangleq \mathbb{E}\left[\|\bar x_k - x^*\|^2\right]$, $\mbox{err}_{2,k}\triangleq \left[\|\mathbf{x}_k-\mathbf{1} \bar x_k\|^2\right]$, and $\mbox{err}_{3,k}\triangleq \mathbb{E}\left[ \| \mathbf{y}_{k}-\mathbf{1}\bar y_{k}\|^2\right]$ for $k\geq 0$.  Suppose $\gamma_k:=\tfrac{\gamma}{k+\Gamma}$ with $\gamma>0$, 
\begin{align}\label{ineq:Gamma_first_ineq}
&\Gamma \geq \gamma \sqrt{\tfrac{3}{1-\rho_W^2}(\tfrac{1}{b^2}+\tfrac{1}{b\eta})\left(2L^2\rho_W^2+\tfrac{2(b-1)L^2(1+\rho_W^2)\rho_W^2}{1-\rho_W^2}\right)},\notag\\
&\Gamma > \gamma, \text{ and }{\Gamma}\geq   \gamma\left(\min\left\{ \tfrac{2b}{\mu + L}, \tfrac{b\mu}{4(b-1)L^2}\right\}\right)^{-1}.
\end{align}
\noindent (a) Then, there exist positive scalars $\theta_t>0$ for $t=1,\ldots,9$ with $\theta_4<1$ and $\theta_6<1$ such that for all $k \geq 0$ we have
\begin{align*}
& \mbox{err}_{1,k+1} \leq (1-\theta_1\gamma_k)\mbox{err}_{1,k} + \theta_2\gamma_k\mbox{err}_{2,k}+\theta_3\gamma_k^2,\\
&\mbox{err}_{2,k+1} \leq (1-\theta_4)\mbox{err}_{2,k} + \theta_5\gamma_k^2\mbox{err}_{3,k},\\
&  \mbox{err}_{3,k+1} \leq (1-\theta_6)\mbox{err}_{3,k} + \theta_7\mbox{err}_{1,k}+\theta_8\mbox{err}_{2,k}+\theta_9.
\end{align*}

\noindent (b) Let $\gamma >\tfrac{1}{\theta_1}$. Let us define  $\hat{\mbox{err}}_{1} := \Gamma\mbox{err}_{1,0}{{N}}_1$, $ \hat{\mbox{err}}_{2} :=  \Gamma\mbox{err}_{2,0}{{N}}_2$, \\and $ \hat{\mbox{err}}_{3} :=\max\left\{ \Gamma\tfrac{3\theta_9}{\theta_6}{{N}}_3,\mbox{err}_{3,0}\right\}$, where ${{N}}_1,{{N}}_2,{{N}}_3>0$ are given as ${{N}}_1 : = \tfrac{2C_2}{C_3\Gamma- 2C_1C_4C_5}$, ${{N}}_2: =\tfrac{2C_4C_5}{ \Gamma}{{N}}_1$, and ${{N}}_3: =\tfrac{C_5}{ \Gamma}{{N}}_1$, where 
$C_1\triangleq \gamma \theta_2\mbox{err}_{2,0}$, $C_2\triangleq \theta_3\gamma^2$, $C_3 \triangleq  \left(\gamma\theta_1-1\right)\mbox{err}_{1,0}$, $C_4 \triangleq \tfrac{6\theta_9\theta_5\gamma^2}{\mbox{err}_{2,0}\theta_4\theta_6}$, $C_5\triangleq  \tfrac{\theta_7\mbox{err}_{1,0}}{\theta_9}$, and $C_6 \triangleq \tfrac{\theta_8\mbox{err}_{2,0}}{\theta_9}$. Then, if $\Gamma > \max\left\{\sqrt{2C_4C_6},\tfrac{2C_1C_4C_5}{C_3},\tfrac{4}{\theta_4}-1\right\}$ and $\eta := \tfrac{b}{2} \left(\tfrac{1- \rho_W^2}{\rho_W^2}\right)$, then we have
\begin{equation}\label{eqn:eik}\boxed{\begin{aligned}
 \mbox{err}_{1,k} \leq  \tfrac{\hat{\mbox{err}}_{1}}{k+\Gamma}, \qquad \mbox{err}_{2,k} \leq \tfrac{\hat{\mbox{err}}_{2}}{(k+\Gamma)^2}, \qquad   \mbox{err}_{3,k} \leq \hat{\mbox{err}}_{3},
\end{aligned}}
\end{equation}
for all
\begin{align}\label{eqn:k_cond}
		k > \gamma\sqrt{\tfrac{{\left(\tfrac{1}{b^2}+\tfrac{2\rho_{W}^2}{b^2(1-\rho_{W}^2)}\right)\left(2L^2\rho_W^2+\tfrac{2(b-1)L^2(1+\rho_W^2)\rho_W^2}{1-\rho_W^2}\right)}}{(1+\rho_W^2)/{2}}}- \Gamma.
	\end{align} 
\end{theorem}
\begin{proof} (a) Consider  Proposition \ref{prop:recursions}. It suffices to show that $0< 1-\theta_6<1$. From Proposition \ref{prop:recursions}(c) we have
\begin{align*}
	1-\theta_6 =  (1+b^{-1}\eta)\rho_W^2+\gamma_k^2&(\tfrac{1}{b^2}+\tfrac{1}{b\eta})\left(2L^2\rho_W^2\right. \nonumber\\
	&\left.+\tfrac{2(b-1)L^2(1+\rho_W^2)\rho_W^2}{1-\rho_W^2}\right).
\end{align*}
It can be observed that $0<1-\theta_6$. It remains to show that $1-\theta_6<1$, that is, we need to show
\begin{align}\label{eqn:theta_6}
	(1+b^{-1}\eta)\rho_W^2+\gamma_k^2&(\tfrac{1}{b^2}+\tfrac{1}{b\eta})\left(2L^2\rho_W^2\right.\nonumber \\
	&\left.+\tfrac{2(b-1)L^2(1+\rho_W^2)\rho_W^2}{1-\rho_W^2}\right) < 1.
\end{align}
From the condition on $k$ in equation \eqref{eqn:k_cond}, we have
\begin{align*}
	\left(\tfrac{k+\Gamma}{\gamma}\right)^2 > \tfrac{{\left(\tfrac{1}{b^2}+\tfrac{2\rho_{W}^2}{b^2(1-\rho_{W}^2)}\right)\left(2L^2\rho_W^2+\tfrac{2(b-1)L^2(1+\rho_W^2)\rho_W^2}{1-\rho_W^2}\right)}}{(1+\rho_W^2)/{2}}.
\end{align*}
This can also be written as
\begin{align*}
	\left(\tfrac{\gamma}{k+\Gamma}\right)^2 < \tfrac{(1+\rho_W^2)/{2}}{{\left(\tfrac{1}{b^2}+\tfrac{2\rho_{W}^2}{b^2(1-\rho_{W}^2)}\right)\left(2L^2\rho_W^2+\tfrac{2(b-1)L^2(1+\rho_W^2)\rho_W^2}{1-\rho_W^2}\right)}}.
\end{align*}
Substituting $\eta = \tfrac{b}{2} \left(\tfrac{1- \rho_W^2}{\rho_W^2}\right)$ and from the definition  of the stepsize sequence $\gamma_{k}=\tfrac{\gamma}{k+\Gamma},$ we have
\begin{align*}
	&\gamma_k^2{\left(\tfrac{1}{b^2}+\tfrac{1}{b\eta}\right)\left(2L^2\rho_W^2+\tfrac{2(b-1)L^2(1+\rho_W^2)\rho_W^2}{1-\rho_W^2}\right)}\\ &< \tfrac{(1+\rho_W^2)}{2}
	= 1-\tfrac{(1-\rho_W^2)}{2}=1 - \tfrac{\eta\rho_{W}^2}{b}
\end{align*}
Therefore, we have
\begin{align*}
 \tfrac{\eta\rho_{W}^2}{b}+	\gamma_k^2{\left(\tfrac{1}{b^2}+\tfrac{1}{b\eta}\right)\left(2L^2\rho_W^2+\tfrac{2(b-1)L^2(1+\rho_W^2)\rho_W^2}{1-\rho_W^2}\right)} &<1.
\end{align*}

\noindent (b) %
Without loss of generality, assume that $\theta_3$ is arbitrarily large such that ${{N}}_1\geq1$. Consecutively, we can state that ${{N}}_2 \geq 1$ and ${{N}}_3 \geq 1$.
We have
$
\mbox{err}_{1,0} \leq \mbox{err}_{1,0} n_1\leq \tfrac{\hat{\mbox{err}}_{1} }{\Gamma}\leq \tfrac{\hat{\mbox{err}}_{1} }{0 + \Gamma}. 
$ The first inequality in \eqref{eqn:eik} holds true for $k=0$.
Now from the definition of  $\hat{\mbox{err}}_{2}$, we have
$
\mbox{err}_{2,0} \leq \mbox{err}_{2,0} {{N}}_2\leq \tfrac{\hat{\mbox{err}}_{2} }{\Gamma^2}\leq \tfrac{\hat{\mbox{err}}_{2} }{(0 + \Gamma)^2}
$. This implies the second inequality in \eqref{eqn:eik} holds for $k=0$.
Further, for  $\mbox{err}_{3,0}$, we have
$\mbox{err}_{3,0} \leq \hat{\mbox{err}}_{3}.$
Therefore, the third inequality in \eqref{eqn:eik} holds true for $k=0$.
Now let the induction hypothesis holds true for  some $k\geq 0$. 
From the  definition of $n_1$, we have
$C_2 \leq\left(C_3\Gamma - 2C_1C_4C_5\right){{N}}_1.$
Therefore, we have $ C_1{{N}}_2\Gamma +C_2 \leq C_3{{N}}_1\Gamma.$
Next, substituting the values of $C_1$, $C_2$, and $C_3$, in the above, we  have
\begin{align*}
\tfrac{\hat{\mbox{err}}_1}{\left(k+\Gamma\right)^2} \leq \tfrac{\gamma\theta_1\hat{\mbox{err}}_1}{\left(k+\Gamma\right)^2} - \tfrac{\gamma \theta_2\hat{\mbox{err}}_2}{\left(k+\Gamma\right)^2}-\tfrac{\theta_3\gamma^2}{\left(k+\Gamma\right)^2} .
\end{align*}
We have $\gamma \geq \gamma_k$. Substituting in the above
\begin{align*}
\tfrac{\hat{\mbox{err}}_1}{\left(k+\Gamma\right)^2} \leq \tfrac{\gamma\theta_1\hat{\mbox{err}}_1}{\left(k+\Gamma\right)^2} - \tfrac{\gamma_k \theta_2\hat{\mbox{err}}_2}{\left(k+\Gamma\right)^2}-\tfrac{\theta_3\gamma^2}{\left(k+\Gamma\right)^2}.
\end{align*}
Further, by bounding $\tfrac{\hat{\mbox{err}}_1}{(k+\Gamma)(k+\Gamma+1)} \leq\tfrac{\hat{\mbox{err}}_1}{(k+\Gamma)^2}$, we obtain
$
\tfrac{\hat{\mbox{err}}_1}{\left(k+\Gamma\right)} - \tfrac{\hat{\mbox{err}}_1}{\left(k+\Gamma +1\right)} \leq \tfrac{\gamma\theta_1\hat{\mbox{err}}_1}{\left(k+\Gamma\right)^2} - \tfrac{\gamma_k \theta_2\hat{\mbox{err}}_2}{\left(k+\Gamma\right)^2}-\tfrac{\theta_3\gamma^2}{\left(k+\Gamma\right)^2} .
$ By induction hypothesis, and Theorem\ref{thm:rate}(a),  the first inequality of \eqref{eqn:eik} holds for $k+1$.
%
Next, from the definition of ${{N}}_2$, we have
$
{{N}}_2  \geq \tfrac{C_4C_5}{ \Gamma}{{N}}_1 \geq C_4 {{N}}_3.
$
Substituting for $C_4$ and rearranging the terms, we obtain
$
\hat{\mbox{err}}_3 \leq \tfrac{\theta_4}{2\theta_5\gamma^2}\ \hat{\mbox{err}}_2 \leq \tfrac{1}{2\theta_5\gamma^2}\left(\theta_4 - \tfrac{\theta_4}{2}\right)\ \hat{\mbox{err}}_2.
$
From $\tfrac{2\Gamma +1}{\left(\Gamma +1\right)^2} \leq \tfrac{\theta_4}{2}$, the preceding inequality becomes
$
 \tfrac{2\Gamma +1}{\left(\Gamma +1\right)^2} \hat{\mbox{err}}_2 \leq \theta_4  \hat{\mbox{err}}_2  - 2\theta_5\gamma^2\hat{\mbox{err}}_3.
$ Further, from  $\tfrac{2k + 2\Gamma +1}{\left(k + \Gamma +1\right)^2} \leq \tfrac{2\Gamma +1}{\left(\Gamma +1\right)^2}$, we have  
$
\tfrac{\hat{\mbox{err}}_2}{\left(\Gamma +1\right)^2} - \tfrac{\hat{\mbox{err}}_2}{\left(k+\Gamma +1\right)^2} \leq \tfrac{\theta_4  \hat{\mbox{err}}_2}{\left(\Gamma +1\right)^2}  - \tfrac{2\theta_5\gamma^2\hat{\mbox{err}}_3}{\left(\Gamma +1\right)^2}$. 
By induction hypothesis and Theorem \ref{thm:rate} (a), we show the
 second inequality of  \eqref{eqn:eik} holds for $k+1$. 

Next, from the definition of $n_3$, we have
${{N}}_3 \Gamma \geq C_5{{N}}_1.$
Substituting value of $C_5$ and rearranging the terms, we obtain
$
 \tfrac{3 \theta_7}{\Gamma} \hat{\mbox{err}}_1 \leq \theta_6 \hat{\mbox{err}}_3.
$ Now, from the upper bound of $\Gamma$, we have $\Gamma^2 \geq 2C_4C_6$. From this, we have $
\Gamma^2 \left(\tfrac{C_5}{ \Gamma}{{N}}_1\right) \geq 2C_4C_6 \left(\tfrac{C_5}{ \Gamma}{{N}}_1\right) \geq C_6\left(\tfrac{2C_4C_5}{ \Gamma}{{N}}_1\right)
$. Substituting for $C_6$ and rearranging terms, we have
$
\tfrac{3\theta_8}{\Gamma^2} \hat{\mbox{err}}_2 \leq \theta_6\hat{\mbox{err}}_3.
$ From  the preceding two results and $3\theta_9 \leq \theta_6 \hat{\mbox{err}}_3$, we have 
$
\left(1-\theta_6\right) \hat{\mbox{err}}_3+\tfrac{\theta_7}{k + \Gamma} \hat{\mbox{err}}_1 + \tfrac{\theta_8}{ \left(k +\Gamma\right)^2} \hat{\mbox{err}}_2+\theta_9 \leq  \hat{\mbox{err}}_3.
$ By induction hypothesis and Theorem \ref{thm:rate} (a), we 
conclude that the third inequality of \eqref{eqn:eik} holds for $k+1$. 
\end{proof}

\section{Numerical Experiments} \label{sec:num_expt}
We provide numerical experiments for comparing the performance of Algorithm \ref{algorithm:RB-DSGT} with other gradient tracking schemes. For the experiments, we consider the distributed regularized logistic regression loss minimization problem. Let the data be denoted by $\mathcal{D}\triangleq \{\left(u_j,v_j\right)\in \mathbb{R}^n\times \{-1,+1\}\mid j \in \mathcal{S}\}$ where $\mathcal{S}\triangleq \{1,\ldots,s\}$ denotes the index set. Let $\mathcal{S}_i$ denote the data locally known to agent $i$ where $\cup_{i=1}^m\mathcal{S}_i =\mathcal{S}_{\tiny\mbox{train}}$. The problem above can be formulated as $\text{min}\txsum_{i=1}^m f_i(x)$ where we define local functions $f_i$ as 
\begin{align*} 
	f_i(x) \triangleq \tfrac{1}{S_{\tiny\mbox{train}}} \txsum_{j \in \mathcal{S}_i}\ln \left(1+\exp\left(-v_j u_j^Tx\right)\right)+\tfrac{\mu}{2m}\|x\|^2.
\end{align*}
Here, $u_j \in \mathbb{R}^n$ denotes attributes, and $v_j \in \{-1,1\}$ for $j \in \mathcal{S}_i$ denotes the binary label of the $j^{th}$ data point.

We simulate the proposed distributed randomized block stochastic gradient tracking method (DRBSGT) algorithm on a system consisting of $m$ agents. We provide a comparison of suboptimality and consensus metrics with the two other existing methods namely, distributed stochastic gradient tracking (DSGT) \cite{Pu_Nedic_20} and adapt then combine (ATC), a variant of block distributed Successive cONvex Approximation algorithm over Time-varying digrAphs (block SONATA) for convex regimes \cite{Notarnicol_20}.

\noindent \textbf{Setup.} The simulations are performed on two data sets with $m$ agents. We use the complete and the ring graph structure to represent the communication among the agents. We implement the simulations on MNIST and Synthetic data set for $m = 5$ and $m = 5, 10 $, respectively. The MNIST data set consist of $50,000$ labels and $784$ attributes, whereas the Synthetic data set has $10,000$ labels and $10,000$ attributes with a Gaussian distribution with mean as $5$ and standard deviation as $0.5$. We consider different parameters for different data sets mentioned in Table \ref{table:Parameters}. Further, we use $\gamma = 1e+1$, $\Gamma = 1e+4$,  $\mu = 1e-1$, and the batch size for computing gradient from each agent $\epsilon = 1e+2$ for both data sets. Taking into account the stochasticity involved in DRBSGT and DSGT schemes, we have obtained different sample paths in our implementations. 
The highlighted areas in  in the plots in Figure \ref{fig:Alg_Output} represent the confidence intervals. We provide 90\% confidence intervals on the errors of the sample paths for each setting in Table \ref{table:CI_compare}. We choose the total sample paths for MNIST and Synthetic data sets as $5$ and $10$, respectively. 

\begin{table}[t] 
\tiny
	\renewcommand\thetable{2}
	\captionsetup{labelformat=empty}
	\captionof{table}{ \footnotesize \green{{Table 2}: Parameters used in the numerical experiments}}
	\centering{\green{
		\setlength\tabcolsep{4pt}
		\begin{tabular}{|l|c|c|c|}
			\hline
			&MNIST$($$m$$=$$5$$)$&Synthetic$($$m$$=$$5$$)$&Synthetic$($$m$$=$$10$$)$\\
			\hline
			\multirow{5}{4em}{DRBSGT} &  $n =784 $& $n =1e+4$ & $n =1e+4$\\
			&$ |S_{\tiny\mbox{train}}| = 5e+4$ & $ |S_{\tiny\mbox{train}}| = 1e+4$ & $ |S_{\tiny\mbox{train}}| = 1e+4$\\ 
			&$\mu = 1$$e$$-$$1$  &$\mu = 1$$e$$-$$1$  &$\mu = 1$$e$$-$$1$\\ 
			&$\epsilon = 1$$e$$+$$2$ &$\epsilon = 1$$e$$+$$2$ &$\epsilon = 1$$e$$+$$2$\\
			& $b = 14$  & $b = 100$& $b = 100$\\ 
			\hline
			\multirow{4}{3em}{DSGT} &   $n =784 $& $n =1e+4$ & $n =1e+4$\\
			&$|S_{\tiny\mbox{train}}| = 5e+4$ & $|S_{\tiny\mbox{train}}|= 1e+4$ & $|S_{\tiny\mbox{train}}| = 1e+4$\\
			&$\mu = 1$$e$$-$$1$ &$\mu = 1$$e$$-$$1$  &$\mu = 1$$e$$-$$1$\\  
			&$\epsilon = 1$$e$$+$$2$ &$\epsilon = 1$$e$$+$$2$&$\epsilon= 1$$e$$+$$2$\\
			\hline
			\multirow{4}{3em}{ATC} &  $n =784 $& $n =1e+4$ & $n =1e+4$\\
			&$|S_{\tiny\mbox{train}}| = 5e+4$ & $|S_{\tiny\mbox{train}}|= 1e+4$ & $|S_{\tiny\mbox{train}}| = 1e+4$\\
			&$\mu = 1$$e$$-$$1$  &$\mu = 1$$e$$-$$1$  &$\mu = 1$$e$$-$$1$ \\ 
			& $b = 14$  & $b = 100$& $b = 100$\\ 
			\hline
	\end{tabular} }}
	\label{table:Parameters}
	\vspace{-.2in}
\end{table}  

\begin{table}[t]
\tiny
	\renewcommand\thetable{3}
	\captionsetup{labelformat=empty}
	\captionof{table}{\scriptsize \green{{Table 3}: $90\%$ CIs for objective values of the three schemes (C: Complete, R: Ring})}
	\centering{			
		\green{
		\setlength\tabcolsep{2pt}
		\begin{tabular}{| l | c |  c | c |} 
			\hline
			& MNIST, $m=5$ &  Synthetic, $m=5$& Synthetic, $m=10$ \\
			\hline
			DRBSGT (C)&$[1.05$$e$$+$$1, 1.06$$e$$+$$1]$&$[9.28$$e$$+$$0, 9.32$$e$$+$$0]$&$[5.07$$e$$+$$0, 5.15$$e$$+$$0]$\\
			\hline
			DRBSGT (R)&$[1.05$$e$$+$$1, 1.06$$e$$+$$1]$&$[9.28$$e$$+$$0, 9.31$$e$$+$$0]$&$[5.10$$e$$+$$0, 5.12$$e$$+$$0]$ \\
			\hline 
			DSGT (C)&$[4.98$$e$$+$$0, 5.47$$e$$+$$0]$&$[6.51$$e$$+$$1, 14.49$$e$$+$$1]$&$[$$14.46$$e$$+$$1, 31.16$$e$$+$$1]$ \\
			\hline
			DSGT (R)&$[5.06$$e$$+$$0, 5.22$$e$$+$$0]$&$[4.62$$e$$+$$1, 9.89$$e$$+$$1]$&$[2.22$$e$$+$$1, 6.01$$e$$+$$1]$\\
			\hline
			ATC (C)&$8.08$$e$$+$$3$&$2.49$$e$$+$$4$&$2.51$$e$$+$$4$\\
			\hline
			ATC (R)&$8.08$$e$$+$$3$&$2.49$$e$$+$$4$&$2.51$$e$$+$$4$\\
			\hline
	\end{tabular}}}
	\label{table:CI_compare}
	\vspace{-.2in}
\end{table}

\noindent \textbf{Insights.}  In Figure \ref{fig:Alg_Output}, the proposed DRBSGT scheme converges in both MNIST and Synthetic data sets, considering both the suboptimality and consensus metrics. We observe that DRBSG performs well compared to its counterparts. Note that Figure \ref{fig:Alg_Output} presents the performance of the schemes with respect to the number of local gradient evaluations. This perhaps explains the slow convergence of ATC as the scheme is deterministic and cycles through all the blocks at each iteration. From Table \ref{table:CI_compare} and Figure \ref{fig:Alg_Output}, we observe that when either the number of attributes or the number of agents increases, the performance of DRBSGT improves. We do not observe any significant difference in the performance in terms of the network connectivity structure.

\section{Concluding Remarks}\label{sec:conclusion}
We address a networked distributed optimization problem in stochastic regimes where each agent can only access unbiased estimators of its local gradient mapping.  Motivated by big data applications, we address this problem considering a possibility of large-dimensionality of the solution space where  the computation of the local gradient mappings may become expensive. We develop a distributed randomized block stochastic gradient tracking scheme and provide the  non-asymptotic convergence rates of the order $1/k$ and $1/k^2$ in terms of an optimality metric and a consensus violation metric, respectively.  We validate the performance of the proposed algorithm on the MNIST and a synthetic data set under different settings of the communication graph.









\bibliographystyle{plain}
\bibliography{ref_hdk_fy_jy-v02}
\end{document}